\documentclass{article}
\usepackage{amsmath}
\usepackage{amssymb}
\usepackage{amsthm}
\usepackage{graphicx}
\usepackage{amscd}
\usepackage{epic, eepic}
\usepackage{url}
\usepackage{color}
\usepackage{todonotes}
\usepackage{comment}
\usepackage{hyperref}
\usepackage{enumitem}

\newtheorem{theorem}{\bf Theorem}[section]
\newtheorem{lemma}[theorem]{\bf Lemma}
\newtheorem{corollary}[theorem]{\bf Corollary}
\newtheorem{cor}[theorem]{\bf Corollary}
\newtheorem{proposition}[theorem]{\bf Proposition}
\newtheorem{prop}[theorem]{\bf Proposition}
\newtheorem{problem}[theorem]{\bf Problem}

\newtheorem{conj}[theorem]{\bf Conjecture}

\theoremstyle{definition}
\newtheorem{notation}[theorem]{\bf Notation}

\newtheorem{remark}[theorem]{\bf Remark}
\newtheorem{definition}[theorem]{\bf Definition}

\newcommand{\cS}{\mathcal{S}}

\newcommand{\supp}{\mathrm{Supp}}
\newcommand{\PG}{\mathrm{PG}}
\newcommand{\AG}{\mathrm{AG}}
\newcommand{\GF}{\mathrm{GF}}
\newcommand{\pf}{\mathrm{pf}_{k}}
\newcommand{\F}{\mathbb{F}}
\newcommand{\ff}{\mathbb{F}}
\newcommand{\zz}{\mathbb{Z}}
\newcommand{\eps}{\varepsilon}

\newcommand{\angbra}[1]{\left\langle #1 \right\rangle}

\newcommand{\twobinom}[2]{\begin{bmatrix}#1\\ #2\end{bmatrix}_2}

\usepackage{todonotes}

\usepackage[utf8]{inputenc}

\title{Avoiding intersections of given size in finite affine spaces $\AG(n,2)$}

\author{Benedek Kovács\thanks{ELTE Linear Hypergraphs Research Group, Eötvös Loránd University, Budapest, Hungary. The author is supported by the ÚNKP-23-3 New National Excellence Program of the Ministry for Culture and
		Innovation from the source of the National Research, Development and Innovation Fund.
		E-mail: {\tt benoke98@student.elte.hu}}
	\and Zolt\'an L\'or\'ant Nagy\thanks{ELTE Linear Hypergraphs  Research Group,
		E\"otv\"os Lor\'and University, Budapest, Hungary. The author is supported by the Hungarian Research Grant (NKFIH) No. PD  134953. and No. K. 124950. 	E-mail: {\tt nagyzoli@cs.elte.hu}}
}
\date{}

\begin{document}
	
	\maketitle
	
	\begin{abstract}
		We study the set of intersection sizes of a $k$-dimensional affine subspace and a point set of size  $m \in [0, 2^n]$ of the $n$-dimensional binary affine space $\AG(n,2)$. Following the theme of Erdős, Füredi, Rothschild and T. Sós, we partially determine which local densities in $k$-dimensional affine subspaces are unavoidable in all $m$-element point sets in the $n$-dimensional affine space.\\
		We also show constructions of point sets  for which the intersection sizes with $k$-dimensional affine subspaces takes values from a set of a small size compared to $2^k$. These are built up from affine subspaces and so-called subspace evasive sets. Meanwhile, we improve the best known upper bounds on subspace evasive sets and apply results concerning the canonical signed-digit (CSD) representation of numbers. 
	\end{abstract}
	
	\section{Introduction}
	
	In this paper, we address the following general problem. Let $S$ be a subset of size $m$ of the affine space $\AG(n,q)$. Does there always exist a $k$-dimensional affine subspace which contains exactly $t$ points of $S$? A $k$-dimensional affine subspace will be referred to as a $k$-\textit{flat}, and a $k$-flat which contains exactly $t$ points of a given set $S$ will be called a \textit{$[k,t]$-flat (induced by $S$)}. If for every $m$-element set $S$ in the $n$-dimensional affine space, there is a $k$-flat containing exactly $t$ points of $S$, then we say that the pair $[n,m]$ forces the pair $[k,t]$, or that $t$-sets are unavoidable in $k$-flats. We use the notation $[n,m]_q\rightarrow [k,t]$ for this concept. 
	
	Our main focus will be the case $q=2$, and we will omit the index $q$ except when we wish to refer to arbitrary finite fields.
	
	The graph theoretic analogue of this problem was initiated by 	Erdős, Füredi, Rothschild and T. Sós \cite{EFRS}.
	
	Let $G(n,m)$ denote a graph on $n$ vertices and $m$ edges. 
	Fix a positive integer $k$ and a pair of integers $(n, m)$ such that $0\leq m\leq \binom{n}{2}$. For which $t$ does it hold that 
	any $n$-vertex graph with $m$ edges contains an induced subgraph on $k$ vertices having exactly $t$ edges? Equivalently, we are seeking pairs $(k,t)$ such that $k$-vertex induced subgraphs with $t$ edges are unavoidable in graphs of the form $G(n,m)$. Erdős, Füredi, Rothschild and T. Sós introduced the notation $(n,m)\rightarrow (k,t)$ for the case when this is true.
	
	Their main result showed that forced pairs $(k,t)$ are rare in the following sense. Consider the set $Sp(n;k,t)$ of all edge cardinalities $m$ such that  $(n,m)\rightarrow (k,t)$.
	Its density is the ratio $\frac{|Sp(n;k,t)|}{\binom{n}{2}}$. They proved that the limit superior of this density is bounded from above by $2/3$ apart from a handful cases, and is $0$ for the majority of the pairs $(k,t)$ for fixed and large enough $k$. Erdős, Füredi, Rothschild and T. Sós
	conjectured that in fact it is bounded from above by $1/2$ apart from finitely many pairs $(k,t)$. This was confirmed recently by He, Ma and Zhao \cite{He}. Several related problems have been studied in the last couple of years \cite{Axen, Axen2, Zarb}.
	
	Our problem can be viewed as the $q$-analogue of this problem, in which we investigate whether all subspaces of given dimension $k$ can avoid to have an intersection of size $t$ with an $m$-element set of the space $\AG(n,q)$, which  corresponds to $\ff_q^n$.
	
	Note that problems of similar flavour have gained significant interest in the past and are trending as well in recent years. The celebrated Sylvester-Gallai theorem showed that if a set of $m$ points in the Euclidean $2$-dimensional space is not collinear, then there will be a line which intersects the point set in exactly $2$ points. \\
	Considering the case $q=3$, $k=1$ and $t=3$, we in turn get the famous cap set problem, which asks for the maximum number of points in $\AG(n,3)$ without creating a line, or in other words, without containing a $3$-term arithmetic progression.  There has been a recent breakthrough due to Ellenberg and Gijswijt \cite{Ellen}, building upon the ideas of Croot, Lev and Pach \cite{CLP}, which showed that to avoid complete lines, $|S|/3^n$ has to be exponentially small. The arithmetic removal lemma implies a supersaturation result on the number of complete lines \cite{Fox2}, from which bounds on the case $[1,3]$ translate to similar bounds for the case of complete $k$-flats, i.e., cases $[k,3^k]$. This connection highlights the complexity of the problem.

	There has been significant interest in the case when we want to forbid each intersection of size larger than or equal to $t$ instead of avoiding only $t$-sets in $k$-flats. If there exists a set $S\subseteq \mathrm{AG}(n,q)$ for which all $k$-dimensional affine subspaces contain at most $c$ points of $S$, then $S$ is called \textit{$(k,c)$-subspace evasive}. The importance of such sets is highlighted by its connections to explicit constructions of bipartite Ramsey graphs by Pudlák and Rödl \cite{Rodl} and to list-decodable codes by Guruswami \cite{Guruswami}.
	
	By a standard application of the first moment method, Guruswami obtained random subsets of $\ff^n$ of large size which are $(k,c)$-subspace evasive. (Here and throughout the paper, $\ff$ denotes a finite field.)
	\begin{theorem}[Guruswami, \cite{Guruswami}]\label{Guru}
		For any fixed pair $(k,c)$,	there exists a $(k,c)$-evasive set in $\ff^n$ of size at least $C\cdot |\ff|^{n(1-\frac{2k}{c})}$, where $C>0$ is a constant independent of $n$.
	\end{theorem}
	
	\begin{cor}\label{intro_evasive_cor}
		For any fixed pair $(k,t)$ with $t>1$, there exists a constant $C>0$ for which the following holds: if $m\le C\cdot 2^{n(1-\frac{2k}{t-1})}$, then $t$-sets are avoidable in the $k$-flats of $\AG(n,2)$, i.e., in this case $[n,m]\not \rightarrow [k,t]$.
	\end{cor}
	
	Surprisingly it turned out that the obtained bounds are sharp in a weak sense. Improving the results of Ben-Aroya and Shinkar \cite{Ben}, Sudakov and Tomon proved the following. 
	
	\begin{theorem}[Sudakov and Tomon, \cite{Sudakov}]
		Let $\ff$ be a field, $k\in \mathbb{Z}^+$ and $\eps \in (0, 0.05)$. If $n$ is sufficiently large with respect to $k$, and $S\subseteq \ff^n$ has size $m\geq |\ff|^{n(1-\eps)}$, then $S$ is not $(k, \frac{k-\log_2(1/\eps)}{8\eps})$-subspace evasive. 
	\end{theorem}

	Note however that $S$ being non-$(k,c)$-evasive does not necessarily imply that $S$ contains a $k$-flat with $c+1$ points, except when $c=|\ff|^{n}-1.$ The latter case on the other hand is not covered by the theorem, as $\frac{k-\log_2(1/\eps)}{8\eps}<2^{k-3}$ holds for every $\eps \in (0, 0.05)$.
	
	\subsection{Our main results}
	
	Similarly to the investigation in \cite{EFRS}, we define the set of forcing sizes $m$ with respect to $[k,t]$, which are the sizes for which a $[k,t]$-flat is unavoidable, as follows.
	
	\begin{definition}
		$$Sp(n;k,t):=\{m: [n,m]_q\rightarrow [k,t] \}$$ is called the \textit{set of forcing sizes with respect to $n, k, t$}, and we refer to it as the \textit{$(n;k,t)$-spectrum}.\\
		$$\rho(n;k,t):= \frac{|Sp(n;k,t)|}{|\ff|^n}$$ is the \textit{density of the spectrum}.
	\end{definition}
	
	(We will always assume that $k\le n$, even if it is not specifically mentioned.)
	
	Our aim is to characterize the spectra or at least bound the density of the spectra for various values of $k$ and $t$. Note that from now on, $\ff$ is considered to be the binary field.
	
	In their paper \cite{EFRS}, Erdős, Füredi, Rothschild and T. Sós construct some graph families which only induce $k$-vertex subgraphs of specific sizes. These families  are the combination of graphs which are very dense (cliques, or graphs close to cliques in edit distance) and sparse graphs (forests, or graphs with high girth).
	
	In that aspect, we follow their footsteps and provide point sets which are combinations of disjoint full flats and point sets which hit every $k$-flat in only a small number of points.
	
	These constructions show that the binary representation of $m$ and $t$ will be key in our results concerning sizes $m$ not in the spectrum. 
	
	\begin{theorem}\label{binalak_intro}	 
		If $0<t<2^k$ then $Sp(n;k,t)$ will not contain those values $m<2^n$ which have fewer $1$-digits than the number of $1$-digits of $t$ in the binary representation, denoted by $s_2(t)$. Since $Sp(n;k,t)$ cannot contain $2^n$ either, this means that at least $$1+\sum_{i=0}^{s_2(t)-1} \binom{n}{i}$$ values are missing from the $(n;k,t)$-spectrum.
	\end{theorem}
	
	In fact, another numeral system, the so-called  canonical signed-digit binary representation of integers will also have a central role. We formulate the corresponding Proposition \ref{th_signed} in Section \ref{evasive_lexico_section}, which is the analogue of Theorem \ref{binalak_intro}.
	
	These results can be combined with the former results on $(k, c)$-subspace evasive sets. Indeed, if we take a point set $S_0$ which avoids $k$-dimensional subspaces having intersection of size $r$ for all $r\in [t_1, t_2]$
	with $c\le t_2-t_1$, then the union of $S_0$ and a $(k, c)$-subspace evasive set $S_1$ will avoid $k$-dimensional subspaces with intersections of size $t\in [t_1+c, t_2]$. 
	
	As we have seen in Corollary \ref{intro_evasive_cor}, subsets of large evasive sets provide point sets of various sizes  which do not induce $[k,t]$-flats for large enough $t$. We improve upon the probabilistic construction of Guruswami in order to obtain larger point sets and to apply the result for every $t\geq k+1$.

	\begin{theorem}[General improved lower bound on evasive sets]\label{ujgur_gen}
		Let $k,c\in \zz^{+}$ with $c\ge k+1$. Then for  $n>k$, there exists a $(k,c)$-evasive set in $\ff_q^n$ of size at least $$ K\cdot q^{n\left(1-\frac{k}{c}\right)}$$ for some constant $K$  independent of $n$.
	\end{theorem}
	
	The combination of these yields long intervals of missing values from the spectrum $Sp(n; k,t)$.
	
	\begin{theorem}\label{combination_final}
		Let $k$ and $t$ be fixed integers such that $1\leq t\leq 2^k-1$ and suppose that the difference between the number of nonzero digits of $t$ and $m$ is greater than $\left\lceil \frac12\left(\log_2 c+1\right)\right\rceil$ for some positive integers $m$ and $c$ in the so-called canonical signed-digit binary representation.
		Then $[n,m']\not\to [k,t]$ holds for every $m'\in \left[m-\left(K\cdot \frac{m}{2^n}\cdot 2^{n\left(1-\frac{k}{c}\right)}\right), m+\left(K\cdot \frac{2^n-m}{2^n}\cdot 2^{n\left(1-\frac{k}{c}\right)}\right)\right]$, where $K$ is a positive constant depending only on $k$ and $c$.
	\end{theorem}
	
	Our next goal is to prove results on long intervals the spectrum $Sp(n;k,t)$ must contain for certain special values of $t$. Our first results in this direction concern the case when $t$ is a power of $2$.
	
	\begin{theorem}\label{2hatvany} Suppose that $k>\ell\ge 1$ are integers. Then there exist constants $C>0$ and $0<D<1$ depending on $k,\ell$ such that $[n,m]\rightarrow \left[k, 2^{k-\ell}\right]$ for $m\in   \left[C\cdot2^{n\left(1-\frac{1}{2^{k-\ell-1}}\right)}, D\cdot 2^n\right]$.
		Moreover, for each $\eps>0$ and sufficiently large $n$, $$\rho(n;k,2^{k-\ell})\ge \frac{1-\eps}{2^{\ell-1}}.$$
	\end{theorem}
	
	As for the case $\ell=0$, it can be shown analogously to the proof of Szemerédi's Cube Lemma (see \cite[Corollary 2.1]{Setyawan}) that in order for a set $S\subseteq \ff_2^n$ to induce a $[k,2^k]$-flat, it suffices for $\frac{|S|}{2^n}$ to be exponentially small, and so $\lim_{n\to \infty} \rho(n;k,2^k)=1$. For $\ell=1$, Theorem $\ref{2hatvany}$ implies that $\lim_{n\to \infty} \rho(n;k,2^{k-1})=1$ holds as well, and we have an even stronger result in this case:
	
	\begin{theorem}\label{halfflat} 
		Let $k\ge 2$. Then there exists a constant $C>0$ depending on $k$ such that $[n,m]~\rightarrow~ \left[k, 2^{k-1}\right]$ for $m\in   \left[C\cdot2^{n\left(1-\frac{1}{2^{k-2}}\right)}, 2^n-C\cdot 2^{n\left(1-\frac{1}{2^{k-2}}\right)}\right]$.
	\end{theorem}
	
	A work closely related to the case of full flats ($\ell=0$) is due to Etzion and Vardy \cite{Etzion}, who determined the minimal number of points needed in $\AG(n,2)$ (and in general in $\AG(n,q)$) in order to cover each subspace of a given dimension $k$, under the name of $q$-Turán designs $\mathcal{T}_q(n,k,1)$. Their result $\mathcal{T}_q(n,k,1)=\frac{q^{n-k+1}-1}{q-1}$ highlights how much harder it is to avoid a complete $k$-flat in our affine variant compared to theirs.
	
	As we have seen, if $t$ is not a power of $2$, then we must expect that the spectrum is \textit{scattered}, several values of $m$ are missing from it.
	
	For $k\le 2$ we have complete characterizations of the spectra, except for the case $[k,t]=[2,4]$ which corresponds to Sidon sets of binary spaces. We present the characterization and discuss the results related to the exceptional case in Section \ref{small_cases_section}. 
	
	Finally, we discuss a case when $t$ is the sum of two consecutive powers of $2$. This case is significantly more involved compared to the case of $t=2^{k-\ell}$.
	
	\begin{theorem}\label{3times2power}
		For every pair $(k,\ell)$ of integers with $2\le \ell\le k-1$, the density of integer values $m$ within the interval $ \left[0, \frac{1}{2^{\ell-1}}\cdot 2^n\right]$ for which  $[n,m]\to [k,3\cdot 2^{k-\ell}]$ holds, tends to $1$ as $n\to \infty$. Hence $\rho(n; k, 3\cdot 2^{k-\ell})\geq \frac{1}{2^{\ell-1}}.$
	\end{theorem}
	
	In general, we have the following conjecture.
	
	\begin{conj}\label{maincon} 
		$$\lim_{n\to \infty} \rho(n;k,t)=1.$$		
	\end{conj}
	
	The paper is organized as follows: in Section \ref{small_cases_section} we present notations and some tools we will use, then discuss the cases when $k$ is small. In Section \ref{evasive_lexico_section} we refine the probabilistic bound of Guruswami and prove Theorem \ref{ujgur_gen}. We continue by  collecting the results on sizes $m$ for which $[k,t]$-flats are avoidable. We discuss the case when $t$ is a power of $2$  in Section \ref{power2_flats_section}, and present the proof of Theorem \ref{2hatvany} and \ref{halfflat}.
	Then we relate the additive energy of  a set $S$ with the number of $[2,3]$-flats induced by $S$ in Section \ref{3power2_flats_section}. We prove bounds on the additive energy via 
	good estimates on the size of certain cuts in hypercube graphs, which enables us to 
	derive  Theorem \ref{3times2power}. Finally, the last section is devoted to concluding remarks and open problems. 
	
	\section{Preliminaries, small cases for $k$}\label{small_cases_section}
	
	\subsection{Notations and tools}
	For a pair of integers $a\le b$, $[a,b]$ will be used to refer to the set $\{a, a+1, ..., b\}\subseteq \zz$. 
	We use the standard additive combinatorial notation for the sumset $A+B$ of two subsets $A$ and $B$ of an abelian group $G$ as the set of all sums of an element from $A$ with an element from $B$, i.e., $A + B = \{a+b : a \in A, b \in B\}$. $A-B$ is defined similarly. If one of the sets consists of a single element, we omit the brackets for short.\\
	The complement of a set $S$ will be denoted by $\overline{S}$.
	
	It will be useful to define the \textit{$k$-dimensional profile $\pf(L)$} of a set $L\subseteq \ff_2^n$ as the set of all possible intersection sizes of $L$ and a $k$-flat of $\ff_2^n$. Thus we have  that $t\not \in \pf(L)$ implies $|L|\not \in Sp(n;k,t)$.
	We show a simple but useful lemma concerning the $k$-dimensional profiles.
	
	\begin{lemma}\label{profilos} \ {} 
		\vspace{-0.3cm}
		\begin{enumerate}[label=(\roman*)] 
			\item    Suppose that  $L_1, L_2\subseteq \ff_2^n$ are disjoint sets. Then $\pf(L_1\cup L_2)\subseteq \pf(L_1)+\pf(L_2)$. 
			\item  Suppose that  $L_1\subseteq L_2\subseteq \ff_2^n$.  Then $\pf(L_2\setminus L_1)\subseteq \pf(L_2)-\pf(L_1)$.
		\end{enumerate} 
	\end{lemma}
	
	\begin{proof}
		These immediately follow from the fact that for any $k$-flat $H$, $|H\cap L_i|\in \pf(L_i)$.
	\end{proof}
	
	The next lemma will be useful for inductive arguments.
	\begin{lemma}[Charbit et al., \cite{Charbit}]\label{char}
		Given any subset $S\subseteq \ff_2^n$ of size $m$, there exists a linear subspace $H\le \ff_2^n$ of dimension $n-1$ such that $|H\cap S|\in \left[\frac12m-\frac12\sqrt{m}, \frac12m+\frac12\sqrt{m}\right]$.
	\end{lemma}
	
	In order to apply bounds on the number of certain affine subspaces,  important ingredients will be the Gaussian binomials and their estimates. 
	\begin{definition}
		$ \begin{bmatrix}n\\ k\end{bmatrix}_q$ denotes the number of $k$-dimensional subspaces of an $n$-dimensional  vector space over $\ff_q$. 
	\end{definition}
	
	\begin{cor}\label{numk-flats}
		$q^{n-k} \begin{bmatrix}n\\ k\end{bmatrix}_q$ is the number of $k$-flats in $\AG(n,q).$
	\end{cor}
	
	We will apply bounds on these Gaussian binomials, i.e. on the number of subspaces. 
	As it was proved in \cite{HegerNagy}, we have
	\begin{lemma}[Héger, Nagy, \cite{HegerNagy}]\label{HN}
		\mbox{$ \begin{bmatrix}n\\ k\end{bmatrix}_2<2^{(n-k)k+1}\cdot e^{2/3}$} and 
		$ \begin{bmatrix}n\\ k\end{bmatrix}_q< q^{(n-k)k}\cdot e^{1/(q-2)}$ for $q>2$  and $k>0$.
	\end{lemma}
	
	We will also use the so-called Euler function $\phi(q)=\prod_{k=1}^{\infty} (1-q^k)$ where $q\in \mathbb{C}$, $|q|<1$.
	
	For a graph $G=(V,E)$ and any set $A\subseteq V$, $e(A)$ denotes the number of edges spanned by $A$, and if $A$ and $B$ are disjoint subsets of $V$ then $e(A,B)$ denotes the number of edges between $A$ and $B$.
	
	\subsection{Small cases}
	
	We will first investigate $Sp(n;k,t)$ in cases where $k$ is small. Firstly, we can see that determining the spectrum $Sp(n;k,t)$, or its density, is essentially the same problem as determining $Sp(n;k,2^k-t)$, or its density:
	
	\begin{lemma}\label{complementing}
		If $n,k\ge 1$ are integers and $0\le t\le 2^k$, then $Sp(n;k,t)=2^n-Sp(n;k,2^k-t)$.
	\end{lemma}
	
	\begin{proof}
		Suppose $m\in Sp(n;k,t)$. Now take any set $T\subseteq \ff_2^n$ with $|T|=2^n-m$.  Then $\overline{T}$ is a set of size $m$, hence there exists a $k$-flat $F_k$ such that $|\overline{T}\cap F_k|=t$. Then $|T\cap F_k|=|F_k|-t=2^k-t$, so $T$ induces a $[k,2^k-t]$-flat. So $2^n-m\in Sp(n;k,2^k-t)$, meaning that $m\in 2^n-Sp(n;k,2^k-t)$. The other containment can be seen by symmetry.
	\end{proof}
	
	\begin{proposition} \ \vspace{-1.4mm}
		\begin{enumerate}[label=(\roman*)] 
			\item $Sp(n;1,0)=[0,2^n-2]$,
			\item $Sp(n;1,1)=[1,2^n-1]$,	
			\item $Sp(n;1,2)=[2,2^n]$,
			\item $Sp(n;2,1)=[0, 2^n]\setminus (\{2^n\} \cup \{2^n-2^d: d\in [0,n]\})$,
			\item $Sp(n;2,2)=[2,2^n-2]$,
			\item $Sp(n;2,3)=[0, 2^n]\setminus (\{0\} \cup \{2^d: d\in [0,n]\})$.  
		\end{enumerate}
	\end{proposition} 
	
	\begin{proof} Consider part (i) first and observe that a $[1,0]$-flat consists of any two points outside $S$, which in turn implies the claim. Similarly in part (iii), any two points in $S$ make a $[1,2]$-flat.
		For part (ii), a $[1,1]$-flat consists of any two points such that one of them is in $S$ and the other is not, so a set $S\subseteq \ff_2^n$ induces a $[1,1]$-flat if and only if $1\le |S|\le 2^n-1$.\\
		For parts (iv) and (vi), a nonempty set $S$ induces a $[2,3]$-flat if and only if there exists $x\in \ff_2^n$ and nonzero $a,b\in \ff_2^n$ such that $x, x+a, x+b\in S$ but $x+a+b\not\in S$. This is exactly the condition such that $S$ is not an affine subspace of $\ff_2^n$, hence the sizes forcing a $[2,3]$-flat are exactly those which are not zero and not powers of 2. By applying Lemma \ref{complementing}, we also obtain the sizes forcing a $[2,1]$-flat.
		
		For part (v), we can see that a $[2,2]$-flat needs at least two points in $S$ and at least two not in $S$, so $Sp(n;2,2)\subseteq [2, 2^n-2]$. On the other hand, if $2\le |S|\le 2^n-2$, then take a pair of points $(s,s')$ with $s\in S$ and $s'\not\in S$. Letting $v=s'-s$, subdivide $\ff_2^n$ into the pairs determined by the $\langle v\rangle$-cosets. Note that any two of the pairs form a $2$-flat. Knowing that the pair $\{s,s'\}$ contains exactly one element of $S$ (i.e. is a $[1,1]$-flat), if we found another pair with exactly one element of $S$ then they together would form a $[2,2]$-flat. Assuming there is no other such pair, all other pairs must either be $[1,2]$- or $[1,0]$-flats. But a $[1,2]$-flat and a $[1,0]$-flat also form a $[2,2]$-flat together, so assuming this does not happen, either all other pairs are $[1,2]$-flats (in which case $|S|=2^n-1$), or all other flats are $[1,0]$-flats (in which case $|S|=1$), giving a contradiction in both cases.
	\end{proof}
	
	The determination of the spectrum $Sp(n;2,4)$ (and its complement $Sp(n;2,0)$) is still a challenging problem. This problem has been studied under the name of Sidon sets in binary vector spaces. A subset $S$ of an abelian group $A$ is a \textit{Sidon set}  if the only solutions to the equation $ a + b = c + d$
	with $a, b, c, d \in  A$ are the trivial solutions when $(a, b)$ is a permutation of $(c, d)$. Observe that for $A=\ff_2^n$, $S$ 
	contains a $[2,4]$-flat if and only if it is not a Sidon set. There are known results on Sidon sets in this setting due to Bose, Ray-Chaudhuri, by Tait, Won and by Lindström, which imply the following.
	\begin{proposition}\label{sidon_prop}
		\begin{enumerate}[label=(\roman*)]
			\item (\cite{BR60}, also see \cite{Lin69, Tait}) There exists a constant $C>0$ such that $[n, m]\to [2,4]$ for every $m\ge C\cdot 2^{n/2}$.
			
			\item (\cite{BR60}) For even $n$, the explicit construction $\{(x, x^3): x\in \ff_{2^{n/2}}\}$ shows that $[n, 2^{n/2}]\not\to [2,4]$.
		\end{enumerate}
	\end{proposition}
	
	The complete characterization of the spectrum for the case $k=3$, $t=4$ requires a combination of various tools, including probabilistic methods.
	
	\begin{proposition}
		$Sp(n;3,4)=[4,2^n-4]$.
	\end{proposition}
	\begin{proof}
		It is clear that a $[3,4]$-flat requires at least 4 points in $S$ and at least 4 not in $S$, so $Sp(n;3,4)\subseteq [4,2^n-4]$.
		
		Firstly we show that in the case $4\le |S|\le 2^{n-2}-1$, $S$ induces a $[3,4]$-flat. Assume on the contrary that it does not. Consider the parallel classes of $2$-flats. In any given parallel class, there must be at least one $[2,0]$-flat (otherwise $S$ would have at least $2^{n-2}$ points). If the same class also contains a $[2,4]$-flat, then together with the $[2,0]$-flat they would form a $[3,4]$-flat, which is a contradiction. So $S$ cannot induce any $[2,4]$-flats at all. Take a line $\ell=\{a,b\}$ joining any two points of $S$. The $2$-flats through $\ell$ cover each point outside $\ell$ exactly once, and there cannot be a $[2,4]$-flat among them, so there must be at least one $[2,3]$-flat (since $|S|\ge 3$). Take any $[2,3]$-flat $H$, and consider the parallel class of $H$. If it contains any $[2,1]$-flat, then together with $H$ it would form a $[3,4]$-flat, which is a contradiction. So the class can only contain flats of type $[2,0]$, $[2,2]$ and $[2,3]$. There can be at most one $[2,2]$-flat, as two $[2,2]$-flats again form a $[3,4]$-flat. Now suppose $H'$ is another $[2,3]$-flat in this class. Let $H=\{x, ~x+a, ~x+b, ~x+a+b\}$ with $x,~x+a,~x+b\in S$ and $x+a+b\not\in S$. Then $H'$ can be written as $\{y,~y+a,~y+b,~y+a+b\}$. If $y, ~y+a\in S$ then $\{x,~x+a,~y,~y+a\}$ is a $[2,4]$-flat, which is a contradiction. Similarly if $y, ~y+b\in S$ we get a contradiction, and also if $y+a, ~y+b\in S$ then we would have the $[2,4]$-flat $\{x+a, ~x+b, ~y+a, ~y+b\}$. So all in all, out of the points $y, ~y+a, ~y+b$ at most one can be in $S$, so $H'$ contains at most 2 points of $S$ and not $3$, which is a contradiction. So the class cannot contain a $[2,3]$-flat different than $H$. Summing up all points in the parallel class, we get that $|S|\le 5$.
		
		Any set of $4$ points is contained in a $3$-dimensional affine subspace, resolving the case $|S|=4$. Finally if $|S|=5$, then in the parallel class of $H$, there must be a $[2,2]$-flat $H'$. Letting $H=\{x,~x+a,~x+b,~x+a+b\}$ with $x,~x+a,~x+b\in S$ and $x+a+b\not\in S$, and $H'=\{y,~y+a,~y+b,~y+a+b\}$, one can easily see that no matter which two points of $H'$ are contained in $S$, $H\cup H'$ will always contain a $[2,4]$-flat induced by $S$, which is a contradiction.
		
		This concludes the proof that any $S$ with $4\le |S|\le 2^{n-2}-1$ induces a $[3,4]$-flat.
		
		Now suppose that $4\le |S|\le 2^n-4$, and let us show that $S$ induces a $[3,4]$-flat. We use induction on $n$. For $n=3$, $|S|=4$ so the whole space is a $[3,4]$-flat, while we have verified the case $n=4$ via exhaustive search. From now on we assume that $n\ge 5$. The case $4\le |S|\le 2^{n-2}-1$ is already covered, and by Lemma \ref{complementing}, we are also done in the case $3\cdot 2^{n-2}+1\le |S|\le 2^n-4$.  Thus the remaining case is when $2^{n-2}\le |S|\le 3\cdot 2^{n-2}$ holds. By complementing, we can actually assume that $2^{n-2}\le |S|\le 2^{n-1}$. We apply Lemma \ref{char} of Charbit et al. \cite{Charbit} on $S$, where either $H$ or $\overline{H}$ will be an $n-1$-flat $A$ with $|S\cap A|\in \left[\frac12m, \frac12m+\frac12\sqrt{m}\right]$. Since $m\in [2^{n-2},2^{n-1}]$, this means that $|S\cap A|\in \left[2^{n-3}, 2^{n-2}+2^{\frac12n-\frac32}\right]$. Here $n\ge 3$ implies that $\frac12n-\frac32\le n-3$, so $|S\cap A|\in [2^{n-3}, 3\cdot 2^{n-3}]$. Since $n\ge 5$, we have $2^{n-3}\ge 4$ and $3\cdot 2^{n-3}\le 2^{n-1}-4$, so $|S\cap A|\in [4, 2^{n-1}-4]$. By the induction hypothesis, the statement is true in the $n-1$-dimensional flat $A$, so $S$ induces a $[3,4]$-flat in $A$.
	\end{proof}
	
	\begin{remark}
		We have seen that $Sp(n;k,2^{k-1})=[2^{k-1}, 2^n-2^{k-1}]$ holds for $k=1,2,3$. However this does not hold for larger $k$. In fact, for $k\ge 4$ there does not exist a constant $C$ independent of $n$ such that $[n,C]\to [k,2^{k-1}]$ holds for all $n\ge k$, since for $n\ge \max(C-1,k+1)$ we can pick a subset $S$ of size $C$ in general position, and then $k$-flats only contain at most $k+1$ points of $S$ while $k+1<2^{k-1}$ for $k\ge 4$. As we will point out in Theorem \ref{ujgur} in the next section, even the  $(k,k+1)$-evasive sets can be exponentially large in $n$.
	\end{remark}
	
	\section{Avoidable sizes with $k$-dimensional subspaces}\label{evasive_lexico_section}
	
	In this section, we show two types of constructions which have (several) missing values in their $k$-dimensional profile. These in turn provide a relatively large set of missing values from the $(n;k,t)$-spectrum. Moreover we point out that these constructions can be combined via disjoint union and difference, leading to even more missing values.
	
	The first type corresponds to evasive sets, which were mentioned in the introduction. Here we refine the probabilistic bounds of Guruswami. Then we continue by constructing $m$-sets which have a profile of relatively small size  provided that $m$ has small support in its binary or in its canonical signed-digit binary representation.
	
	\subsection{Evasive sets}
	
	As we mentioned in the introduction, the existence of $(k,c)$-evasive sets in turn provides sets which avoid every intersection size greater than $c$. The random construction of Guruswami showed that we can obtain a $(k,c)$-evasive set on at least $2^{n(1-\frac{2k}{c})}/2^{k+1}$ points. Below we refine his argument using alteration.
	
	\begin{theorem}\label{ujgur}
		Let $k,c\in \zz^{+}$ with $c\ge k+1$. Then for  $n>k$, there exists a $(k,c)$-evasive set in $\ff_2^n$ of size at least $$\lfloor K\cdot 2^{n\left(1-\frac{k}{c}\right)}\rfloor-1$$ where $K=K(k,c):=\frac{c}{c+1}\cdot 2^{k(k+1)/c}\cdot \left(2e^{2/3}(c+1)\binom{2^k}{c+1}\right)^{-\frac{1}{c}}$.
	\end{theorem}
	
	\begin{proof}
		To construct a $(k,c)$-evasive set, we apply the following procedure. First we suitably choose a value $m\in [c, 2^n]$. Then we take a random point set of size $ m$ in $\ff_2^n$, and if we find a set of $c+1$ points within a $k$-flat, we erase one of those points.\\
		Then the expected value of the number of erased points can be bounded from above by \begin{equation*}
			2^{n-k} \begin{bmatrix}n\\ k\end{bmatrix}_2
			\binom{2^k}{c+1}\frac{\binom{{2^n}-c-1}{m-c-1}}{\binom{2^n}{m}}<2^{(n-k)(k+1)+1}e^{2/3}\binom{2^k}{c+1}\left(\frac{m}{2^n}\right)^{c+1}.
		\end{equation*}
		Indeed, the probability that a given $k$-flat contains at least $c+1$ points of an $m$-element random point set in $\ff_2^n$ is less than   $\binom{2^k}{c+1}\frac{\binom{{2^n}-c-1}{m-c-1}}{\binom{2^n}{m}}$, and we can apply the upper bound given in Lemma \ref{HN} on the  number of $k$-flats, which  is $2^{n-k} \begin{bmatrix}n\\ k\end{bmatrix}_2$, see Corollary \ref{numk-flats}.
		
		The remaining points form a $(k,c)$-evasive set of size at least 
		$$m-2^{(n-k)(k+1)+1}e^{2/3}\binom{2^k}{c+1}\left(\frac{m}{2^n}\right)^{c+1}$$ which we want to maximize in $m$. Note that this is a concave function of $m$, maximized when its derivative is $0$. This happens when $C(n,k,c)(c+1)m^{c}=1$, where
		$$C(n,k,c):=2^{(n-k)(k+1)+1}e^{2/3}\binom{2^k}{c+1}\left(\frac{1}{2^n}\right)^{c+1}.$$
		
		Hence the maximum of this function over $\mathbb{R}$ is attained at $$m= 2^{n\left(1-\frac{k}{c}\right)}\cdot 2^{k(k+1)/c}\cdot \left(2e^{2/3}(c+1)\binom{2^k}{c+1}\right)^{-\frac{1}{c}},$$
		
		while the maximum itself is $\frac{c}{c+1}m$. As $m$ must be an integer, we get that the expected value is at least $\frac{c}{c+1}m-1$ for $$m= \left\lfloor 2^{n\left(1-\frac{k}{c}\right)}\cdot 2^{k(k+1)/c}\cdot \left(2e^{2/3}(c+1)\binom{2^k}{c+1}\right)^{-\frac{1}{c}}
		\right\rfloor. $$ Note that the main term is $2^{n\left(1-\frac{k}{c}\right)}$ while its multiplicative factor is independent of $n$, moreover using Stirling's approximation on the constant, we can obtain that the number of remaining points is at least $$2^{(n-k)(1-\frac{k}{c}) }\cdot\frac{c^2}{2e(c+1)}-1$$ for $c\geq k+1\geq 2.$ This result refines the Theorem \ref{Guru}  of Guruswami.
	\end{proof}
	
	\begin{remark}
		Note that in the case $c\le k$, any set of $c+1$ points is contained in a $c$-flat thus a $(k,c)$-evasive set can not contain more than $c$ points. However, for $c>k$ we already have an exponential lower bound on the size of $(k,c)$-evasive sets in $\ff_q^n$. 
	\end{remark}
	
	If $q>2$, the same proof confirms the following generalization. The only difference is that we have a different constant for the bound on $\begin{bmatrix}n\\ k\end{bmatrix}_q$ compared to $\begin{bmatrix}n\\ k\end{bmatrix}_2$ in Lemma \ref{HN}.
	
	\begin{theorem}[General improved lower bound on evasive sets]\label{ujgur2}
		Let $k,c\in \zz^{+}$ with $c\ge k+1$ and $q>2$. Then for  $n>k$, there exists a $(k,c)$-evasive set in $\ff_q^n$ of size at least $$\lfloor K\cdot q^{n\left(1-\frac{k}{c}\right)}\rfloor-1,$$ where $K=K(k,c):=\frac{c}{c+1}\cdot 2^{k(k+1)/c}\cdot \left(2 e^{1/(q-2)}(c+1)\binom{q^k}{c+1}\right)^{-\frac{1}{c}}$.
	\end{theorem}
	
	\subsection{Sets related to binary and canonical signed-digit representations}
	
	To prove Theorem \ref{binalak_intro}, we will use the following construction.
	
	\begin{definition}\label{lexico}
		Let $m\in [0,2^n]$ be an integer. The \textit{lexicographic construction of $m$ points} is defined as follows. Take every point $P=(x_{n-1}, x_{n-2}, \ldots, x_0)$ in $\F_{2}^n$ for which the $n$-digit binary representation $\overline{x_{n-1}x_{n-2}\ldots x_0}$ represents a number smaller than $m$.\\ The lexicographic construction of $m$ points will be denoted by $\mathcal{L}_m$.
		The binary \textit{support} $\supp(m)$ of $m$ is the set of values $d\ge 0$ such that the binary form of $m$ contains a digit 1 at the place $2^d$.\\
		Finally, let $s_2(m):=|\supp(m)|$, i.e.,   the sum of the digits in the binary representation of $m$.
	\end{definition}
	
	The following lemma is a straightforward consequence of the summation procedure, taking into account the decrease of the digit sum at each carry.
	
	\begin{lemma}\label{bin_sum}   For any natural numbers $a_1$, $a_2$, ..., $a_{\ell}$ we have $s_2\left(\sum_{i=1}^{\ell} a_i\right)\le \sum_{i=1}^{\ell} s_2(a_i)$.
	\end{lemma}
	
	\begin{prop}\label{binalak}	 If $s_2(t)>s_2(m)$ then $[n,m]\not \to [k,t]$, and in particular, $\mathcal{L}_m\subseteq \ff_2^n$ does not induce a $[k,t]$-flat. Moreover, if $H$ is a $k$-flat in $\ff_2^n$, then the size of $H\cap \mathcal{L}_m$ can admit at most  $$1+\sum_{j=0}^{s_2(m)} \binom{k}{j}$$ values.
	\end{prop}
	\begin{proof}
		Let $M=s_2(m)$, and let $m=\sum_{i=1}^M 2^{d_i}$, thus $ \supp(m)=\{d_1<d_2<\ldots <d_M\}$. Then $\mathcal{L}_m$ naturally decomposes into $M$ flats $H_1, H_2, \ldots, H_M$ of dimension $d_{1}, \ldots, d_{M}$, respectively, where the flat $H_j$ of dimension $d_{j}$ is given by the equations
		$$x_k=
		\begin{cases}
			1  & \text{for } k>d_j, k\in \supp(m)\\
			0  & \text{for } k>d_j, k\not \in \supp(m)\\
			0  & \text{for } k=d_j.
		\end{cases}$$
		
		Denote $\dim (H\cap H_i)$ by $r_i$. The size of $H\cap \mathcal{L}_m$ can be expressed as  $$\sum_{i:H_i\cap H\neq \emptyset} 2^{r_i}.$$ 
		
		Since the sets $H\cap H_i$  are disjoint subsets of $H$ and of $\mathcal{L}_m$, we conclude  the sum above is at most $\min \{2^k, m\}$. Moreover, if the sum is smaller than $2^k$, then the binary representation of the sum consists of at most $k$ digits, of which at most $M$ is one, in view of Lemma \ref{bin_sum}. Taking into account the possibility $|H\cap \mathcal{L}_m|=2^k$ as well, this completes the proof.
	\end{proof}
	
	Theorem \ref{binalak_intro} follows immediately from Proposition \ref{binalak}.
	
	Not only the binary digit sum can impose conditions on the  spectrum: the lexicographic construction shows further values of $m$ for which $[k,t]$-flats are avoidable. Suppose that $m$ is a difference of two powers of $2$. While $s_2(m)$ might be large, the lexicographic construction of $m$ points shows that $[k,t]$ is avoidable if $t$ cannot be written as a power of $2$ or a difference of two powers of $2$. It is easy to deduce a statement similar to Proposition \ref{binalak}, after recalling the definition and the properties of the canonical signed-digit binary representation.
	
	Let $s^*_2(m)$ denote the number of nonzero values in the \textit{non-adjacent form} or\textit{ canonical signed-digit (CSD) binary representation} of $m$. This  number system is a signed-digit number system using three digits, $0, 1, \overline{1}$,  which minimizes the number of non-zero digits. $\overline{1}$ corresponds to $-1$. Its main advantage compared to the binary form is that CSD reduces the number of operations in a hardware multiplier. A detailed discussion of CSD can be found in \cite{Reit}. We apply the following facts concerning this numerical system.
	
	\begin{prop}\label{CSD_properties} \ {}
		\begin{itemize} 
			\item  The CSD representation of any integer number can be gained easily from the binary form by iterating the following step: starting from the right, going left,
			each sequence of length at least $3$  consisting of a single $0$ in the first position of the sequence and  continued solely by $1$s are replaced by a sequence of the same length, starting with 1, ending with $\overline{1}$ and containing only $0$ in between.
			\item The CSD representation of any integer number is unique.
			\item There are no two consecutive non-zero digits in the CSD representation of any integer number. Moreover, finite sequences built up by starting with a digit $1$ or $\overline{1}$ (as the leftmost digit) and continuing with zero or more blocks of the form $0$, $01$, or $ 0\overline{1}$ correspond uniquely to the nonzero integers via the CSD representation.
			\item For every integer $n$, the CSD representation contains the least number of nonzero digits compared to any signed-digit binary representation of $n$.
		\end{itemize}
	\end{prop}
	
	Table \ref{table:1} presents the CSD representation of the first few positive integers. Note that the digits of the CSD form of $m$ correspond to the coefficients of the powers of $2$ when $m\in \mathbb{Z}^+$ is written (uniquely) as  $m=\sum_{i}\alpha_i 2^{d_i}$ such that $\alpha_i\in \{\pm 1\}$ for all $i$ and no pair of consecutive positive integers appears in the set of exponents of $2$, i.e., $|d_i-d_j|\neq 1$ for all pairs $i,j$.
	
	\begin{center}
		\begin{table}[h!]
			\begin{tabular}{l|l|l|l|l|l|l|l|l|l|l|l|l }
				$n$ &1 &2& 3 &	4	& 5 & 6  & 7  & 8 & 9 & 10 & 11 & 12 \\ \hline
				CSD form of $n$ &1 & 10 & $10\overline{1}$ &	100	& 101  & $10\overline{1}0$  & $100\overline{1}$  & 1000 & 1001  & 1010  & $10\overline{1}0\overline{1}$ &  $10\overline{1}00$
			\end{tabular}\caption{CSD form of positive integers less than 13.}
			\label{table:1}
		\end{table} 
	\end{center}

	We have an analogue of Lemma \ref{bin_sum} for the CSD binary representation, which follows easily from the last property of the CSD representation, formulated above.
	
	\begin{lemma}\label{CSDbin_sum}   For any integer numbers $a_1$, $a_2$, ..., $a_{\ell}$ we have $s_2^*\left(\sum_{i=1}^{\ell} a_i\right)\le \sum_{i=1}^{\ell} s_2^*(a_i)$.
	\end{lemma}

	The CSD representation offers a convenient tool to exclude further values from the spectra $Sp(n;k,t).$
	
	\begin{prop}\label{bin_nonadjacent_alak}	 If $s^*_2(t)>s^*_2(m)$ then $[n,m]\not \to [k,t]$. 
	\end{prop}
	
	\begin{proof}
		The proof follows that of Proposition \ref{binalak}. Take the CSD representation $m=\sum_{i=1}^M \alpha_i 2^{d_i}$ with $d_1<d_2<...<d_M$ and $\alpha_i\in \{\pm 1\}$. We will now provide a construction of $\mathcal{L}_m\subseteq \ff_2^n$ according to this CSD form. For each $1\le j\le M+1$, let $m_j=\sum_{i=j}^M \alpha_i 2^{d_i}$. (Note that $m_1=m$ and $m_{M+1}=0$.) Then one can obtain $\mathcal{L}_m$ by starting from $\mathcal{L}_{m_j}=\emptyset$ for $j=M+1$, and for each $j=M, M-1, ..., 1$ in turn, building $\mathcal{L}_{m_j}$ from $\mathcal{L}_{m_{j+1}}$. Observe that since $2^{d_j}\mid m_{j+1}$, adding the next $2^{d_j}$ points in lexicographical order to $\mathcal{L}_{m_{j+1}}$ (if $\alpha_j=1$), or taking away the greatest $2^{d_j}$ points (if $\alpha_j=-1$), will correspond to adding an $m_j$-flat, or taking it away respectively. This $m_j$-flat will be denoted $H_j$ (for each $M\ge j\ge 1$).
		
		Denote $\dim (H\cap H_i)$ by $r_i$. The size of $H\cap \mathcal{L}_m$ can be expressed then as $$\sum_{i:H_i\cap H\neq \emptyset} \alpha_i 2^{r_i}.$$ By Lemma \ref{CSDbin_sum}, we get that the number of nonzero digits of $|H\cap \mathcal{L}_m|$ is smaller than or equal to that of the CSD representation of $m$.  
	\end{proof}
	
	\begin{prop}\label{th_signed}
		If $0<t<2^k$ then $Sp(n;k,t)$ will not contain those values $m>0$ which have fewer nonzero digits than the number of nonzero digits of $t$ in their respective CSD representation, hence at least $$n+2+\frac12\cdot\sum_{i=2}^{s_2^{*}(t)-1} \left[\binom{n+1-i}{i}2^i+\binom{n+1-i}{i-1}2^{i-1}\right]$$ values are missing from the $(n;k,t)$-spectrum.
	\end{prop}
	
	\begin{proof}
		We have to count those canonical signed-digit representations $\overline{x_nx_{n-1}...x_0}$ of value in $[1,2^n]$ which have $i\in [1,s_2^*(t)-1]$ nonzero digits. If $i=1$ then there are $n+1$ such representations. If $i\ge 2$ then consider two cases. If $x_n=0$ then $\overline{x_n...x_0}$ is a sequence of $i$ blocks $01$ or $0\overline{1}$ and $n+1-2i$ blocks $0$ in some order, but the first two-digit block must be $01$ (so that the number is positive), so there are $\frac12\cdot \binom{n+1-i}{i} 2^i$ such representations. If $x_n=1$ then $\overline{x_{n-1}...x_0}$ is a sequence of $i-1$ blocks $01$ or $0\overline{1}$ and $n+2-2i$ blocks $0$ in some order, but the first two-digit block must be $0\overline{1}$ (so that the number is less than $2^n$), so there are $\frac12\cdot \binom{n+1-i}{i-1}2^{i-1}$ such representations. Accounting also for the case $m=0$, this finishes the proof.
	\end{proof}
	
	\subsection{Combination}
	
	The propositions above can be combined with the former results on $(k, c)$-evasive sets. Indeed, consider a lexicographic construction $\mathcal{L}_m$ which avoids $[k,t]$-flats for all $t\in [t_1, t_2]$. Take a $(k, c)$-subspace evasive set $S$ disjoint from  $\mathcal{L}_m$. Then   $\mathcal{L}_m\cup S $ provides a set of size $m+|S|$ which avoids $[k,t]$-flats for all $t\in [t_1+c, t_2]$.
	
	Applying our improvement Theorem \ref{ujgur} on $(k,c)$-evasive sets, we obtain the following result.
	
	\begin{theorem}\label{combination_addition}
		Let $k$ and $t$ be fixed integers such that $1\leq t\leq 2^k-1$ and suppose that either  $s_2(t)-~ s_2(m)> \lfloor\log_2 c\rfloor+1$ or $s_2^*(t)-~ s_2^*(m)> \left\lceil \frac12\left(\log_2 c+1\right)\right \rceil$ holds for a pair of positive integers $m, c$ with $c>k$.
		Then $[n,m']\not\to [k,t]$ holds for every integer $m'\in \left[m, m+\left(K\cdot \frac{2^n-m}{2^n}\cdot 2^{n\left(1-\frac{k}{c}\right)}\right)\right]$, where $K$ is a positive constant depending only on $k$ and $c$.
	\end{theorem}
	
	\begin{proof}
		First we prove the statement for the digit sum $s_2$. Consider the lexicographic construction $\mathcal{L}_m$ on $m$ points, and a $(k,c)$-evasive set $S$ in $\ff_2^n$ disjoint from $\mathcal{L}_m$.
		
		Lemma \ref{profilos} implies that if the union of $\mathcal{L}_m$ and  $S$ induces  a $[k,t]$-flat, that is,  $t\in \pf(\mathcal{L}_m\cup S)$, then $t\in \pf(\mathcal{L}_m)+\pf(S)$  holds as well.
		$s_2(z)\leq s_2(m)$ holds for every $z\in \pf(\mathcal{L}_m)$. On the other hand, every $k$-flat intersects $S$ in at most $c$ points, thus $\pf(S)\subseteq [0, c]$. Consequently, $s_2(t)\leq s_2(m)+\lfloor \log_2 c\rfloor+1$ for every $t\in \pf(\mathcal{L}_m\cup S)$ in view of the subadditive property of the digit sum function $s_2$, Lemma \ref{bin_sum}. (Note that any element $t'\in [0,c]$ satisfies $s_2(t')\le \lfloor \log_2 c\rfloor+1$.)
		We thus get that if $s_2(t)> s_2(m)+\lfloor \log_2 c\rfloor+1$, then we can add a $(k,c)$-evasive set to $\mathcal{L}_m$ to gain a set which does not induce a $[k,t]$-flat. If $s_2(t)-s_2(m)> \lfloor\log_2 (k+1)\rfloor+1$, the size of this subspace evasive set can be exponentially large in $n$, according to Theorem \ref{ujgur}. Indeed, we can find a set of size $K\cdot 2^{n\left(1-\frac{k}{c}\right)}$, and we can keep at least a fraction $\frac{2^n-m}{2^n}$ of it to obtain a disjoint set from $\mathcal{L}_m$. This follows from the fact that we might consider any translate $w+\mathcal{L}_m$ of the lexicographic construction for $w\in \ff_2^n$, and on average these contain a fraction $\frac{m}{2^n}$ of the evasive set.\\
		All in all, we get a set which avoids $[k,t]$-flats and has size $m+s$ for each integer $$0\le s\le \frac{2^n-m}{2^n}\cdot \left(\left\lfloor K(k,c)\cdot  2^{n\left(1-\frac{k}{c}\right)}\right\rfloor-1\right),$$ where $K(k,c)=\frac{c}{c+1}\cdot 2^{k(k+1)/c}\cdot \left(2e^{2/3}(c+1)\binom{2^k}{c+1}\right)^{-\frac{1}{c}}$.
		
		The proof of the statement for $s_2^*$ is identical, except that now any $t'\in [0,c]$ will satisfy $s_2^*(t')\le \left\lceil \frac12(\log_2 c+1)\right\rceil$, since the number of digits in the CSD of a positive integer $t'$ is at most $\lceil\log_2 t'\rceil+1$, and for any two neighbouring digits, at most one of them is nonzero.
	\end{proof}
	
	We showed that if $s_2(m)$ is small compared to $s_2(t)$, then the integers in a certain right neighbourhood of $m$ will also not be in the spectrum $Sp(n;k,t)$.
	
	A similar construction can be obtained by erasing a $(k, c)$-evasive set $S$ from a lexicographic construction $\mathcal{L}_m$ which avoids $[k,t]$-flats for all $t\in [t_1, t_2]$. Then $\mathcal{L}_m\setminus S$ provides a set of size $m-|S|$ which avoids $[k,t]$-flats for all $t\in [t_1, t_2-c]$.
	
	\begin{theorem}\label{combin_diff} Let $k$ and $t$ be fixed integers such that $1\leq t\leq 2^k-1$ and suppose that $s_2^*(t)-~s_2^*(m)> \left\lceil \frac12\left(\log_2 c+1\right)\right \rceil$ holds for a pair of positive integers $m, c$.
		Then $[n,m']\not\to [k,t]$ holds for every integer $m'\in \left[m-\left(K\cdot \frac{m}{2^n}\cdot 2^{n\left(1-\frac{k}{c}\right)}\right), m\right]$, where $K$ is a positive constant depending only on $k$ and $c$.
	\end{theorem}
	
	\begin{proof}
		Consider the lexicographic construction $\mathcal{L}_m$, and a $(k,c)$-evasive set $S$ in $\ff_2^n$ which is a subset of $\mathcal{L}_m$.
		
		By Lemma \ref{profilos}, if $t\in \pf(\mathcal{L}_m\setminus S)$, then $t\in \pf(\mathcal{L}_m)-\pf(S)$ too.
		$s_2^*(z)\leq s_2^*(m)$ holds for every $z\in \pf(\mathcal{L}_m)$. On the other hand, every $k$-flat intersects $S$ in at most $c$ points, thus $\pf(S)\in [0, c]$. Consequently, $s^*_2(t)\leq s^*_2(m)+\left\lceil \frac12\left(\log_2 c+1\right)\right \rceil$ for every $t\in \pf(\mathcal{L}_m\setminus S)$ in view of subadditive property of the  function $s_2^*$, Lemma \ref{CSDbin_sum}.
		We thus get that if $s_2^*(t)> s_2^*(m)+\left\lceil \frac12\left(\log_2 c+1\right)\right \rceil$, then we can delete a $(k,c)$-evasive set from $\mathcal{L}_m$ to gain a set which does not induce a $[k,t]$-flat. By Theorem \ref{ujgur}, we can find a $(k,c)$-evasive set of size $K\cdot 2^{n\left(1-\frac{k}{c}\right)}$, and we can keep at least a fraction $\frac{m}{2^n}$ of it to obtain a subset of $\mathcal{L}_m$, similarly to the proof of Theorem \ref{combination_addition}. \\
		So the resulting set avoids $[k,t]$-flats and has size $m-s$ for each $0\le s\le \frac{m}{2^n}\cdot \left(\left\lfloor K(k,c)\cdot 2^{n\left(1-\frac{k}{c}\right)}\right\rfloor-1\right)$.
	\end{proof}
	
	The two theorems above imply Theorem \ref{combination_final}.
	
	\section{Avoiding intersection sizes equal to a power of $2$}\label{power2_flats_section}
	
	In the previous section, we proved that $[k,t]$-flats are easier to avoid if the binary representation of $t$ contains many digits equal to $1$. Here we discuss the case when $s_2(t)$ is $1$ or $2$. The cornerstone will be a simple theorem which states that if a subset of $\ff_2^n$
	contains a large proportion of the vectors of the space then it must contain a large affine subspace as well.
	
	\subsection{Finding $[k,2^k]$-flats}
	
	We will first consider the case of $t=2^k$, which means that we want to find a full affine subspace of dimension $k$ in $S$. Recall that in Section \ref{small_cases_section}, we have discussed the cases $k=1$ and $k=2$, with the case $k=2$ corresponding to Sidon sets. The following theorem extends the result of Bose and Ray-Chaudhuri, presented  in Proposition \ref{sidon_prop}(i).  
	
	\begin{theorem}\label{order2}
		Given integers $n\ge k\ge 1$, and any $m\ge \frac52\cdot 2^{n\left(1-\frac{1}{2^{k-1}}\right)}$, we have $[n,m]\to [k,2^k]$.
	\end{theorem}
	The result points out that for a fixed integer $k$, sets with given positive density will always contain a $k$-flat if $n$ is large enough. We mention that the result and its proof is closely related to Szemerédi's Cube Lemma, see \cite[Corollary 2.1]{Setyawan}. The statement  of Theorem \ref{order2}  follows from the work of Bonin and Qin \cite{Bonin}. We present a proof as it will be referred to during the proofs of the main results of the chapter.
	
	\begin{lemma}\label{seged}
		Suppose that $|S|\ge D\cdot 2^{\alpha n}+1$ for a set $S\subseteq \ff_2^n$, where $\frac12\le \alpha<1$ and $D>0$ are real numbers. Then there exists a nonzero vector $d\in \ff_2^n$ such that the number of unordered pairs of vectors in $S$ with difference $d$ is at least $\frac{D^2}{2}\cdot 2^{(2\alpha-1)n}$.
	\end{lemma}
	
	\begin{proof} Taking the elements of $S$ pairwise, they form $\binom{|S|}{2}$ differences with multiplicity. The number of possible differences is $2^n-1$, so there must exist a difference appearing at least $$\frac{\binom{|S|}{2}}{2^n-1}>\frac{\binom{|S|}{2}}{2^n}\ge \frac{D^2\cdot 2^{2\alpha n}}{2^{n+1}}=\frac{D^2}{2}\cdot 2^{(2\alpha-1)n}$$ times.
	\end{proof}
	
	\begin{proof}[Proof of Theorem \ref{order2}]
		We prove that there exist constants $0<C_k\le \frac52$ for each $k\ge 1$ such that if $|S|\subseteq \ff_2^n$ has $|S|\ge C_k\cdot 2^{n\left(1-\frac{1}{2^{k-1}}\right)}$ then $S$ contains a  $k$-flat. For $k=1$, the result holds with $C_1=2$, as any two points form a $1$-flat.
		
		Now suppose that $k\ge 2$. If we know that $|S|\ge D_k\cdot  2^{n\left(1-\frac{1}{2^{k-1}}\right)}+1$ for a certain constant $D_k>0$, then we can apply Lemma \ref{seged} to get a difference $d_1$ which appears at least $\frac{D_k^2}{2}\cdot 2^{n\left(1-\frac{1}{2^{k-2}}\right)}$ times. We factorize the space by this vector and define a set $S_1\subseteq \ff_2^n/\langle d_1\rangle$ which consists of the $\langle d_1\rangle$-cosets with both elements lying in $S$. Then $|S_1|\ge \frac{D_k^2}{2}\cdot 2^{n\left(1-\frac{1}{2^{k-2}}\right)}$.
		
		Observe that if we find a $[k-1, 2^{k-1}]$-flat induced by $S_1$, then it can be lifted to a $[k,2^k]$-flat induced by $S$. So to find such a $(k-1)$-flat and finish the proof, it suffices to refer to the induction hypothesis for the subset $S_1$ of the $(n-1)$-dimensional space $\ff_2^n/\langle d_1\rangle$. For this, it suffices to have $\frac{D_k^2}{2}\cdot 2^{n\left(1-\frac{1}{2^{k-2}}\right)}\ge C_{k-1}\cdot 2^{(n-1)\left(1-\frac{1}{2^{k-2}}\right)}$.
		
		This is equivalent to $D_k\ge \sqrt{C_{k-1}}\cdot 2^{\frac{1}{2^{k-1}}}$, so we need a value $C_k$ such that $C_k\cdot 2^{n\left(1-\frac{1}{2^{k-1}}\right)}\ge \sqrt{C_{k-1}}\cdot 2^{\frac{1}{2^{k-1}}}\cdot 2^{n\left(1-\frac{1}{2^{k-1}}\right)}+1$ holds for all $n\ge k$. The choice $C_k=\sqrt{C_{k-1}}\cdot 2^{\frac{1}{2^{k-1}}}+\frac{1}{2^{k\left(1-\frac{1}{2^{k-1}}\right)}}$ works.
		
		From $C_1=2$ this gives $C_2=\frac52$, and for $k\ge 3$, if we already know that $C_{k-1}\le \frac52$ then we will get $C_k\le \sqrt{\frac52}\cdot 2^{1/4}+\frac{1}{2^{3\left(1-\frac{1}{2^2}\right)}}\approx 2.09<\frac52$.
	\end{proof}
	
	\begin{remark}
		In fact, in the argument above the recursive formula gives $C_k\to 1$ as $k\to \infty$. From this argument one can also deduce that for each fixed $k\ge 1$ and $\delta>0$, $|S|\ge (2+\delta)\cdot 2^{n\left(1-\frac{1}{2^{k-1}}\right)}$ already implies the existence of a $[k,2^k]$-flat for sufficiently large $n$.
	\end{remark}
	
	\subsection{Finding $[k,2^{k-\ell}]$-flats for almost all values $|S|$ when $S$ has a low positive density}
	
	In the following, we will show that $k$-flats of density exactly $\frac{1}{2^{\ell}}$ can be found for a long interval of values $m$ which contains a positive fraction of all values as $n\to \infty$, where this fraction is $1$ for $\ell=1$.
	
	The proof relies on bounds for flats of smaller dimension, containing a given number of points. 
	
	\begin{notation}
		Suppose that a set $S\subseteq \ff_2^n$ is given. If $\{v_1,v_2,...,v_d\}$ is a set of linearly independent vectors in $\ff_2^n$, and $0\le t\le 2^d$ is an integer, let $F_{d,t}(S; v_1,v_2,...,v_d)$ denote the number of $\langle v_1,v_2,...,v_d\rangle$-cosets in $\ff_2^n$ containing exactly $t$ elements of $S$.\\
		The number of all $[d,t]$-flats induced by $S$ will be denoted by $F_{d,t}(S)$.\\ From now on, we apply the simpler notation 
		$F_{d,t}(v_1,v_2,...,v_d)$ and $F_{d,t}$, respectively, if the set $S$ is clear from the context.
	\end{notation}
	
	We recall the statement of Theorem \ref{halfflat}.
	
	\textbf{Theorem \ref{halfflat}.} \textit{
		Let $k\ge 2$. Then there exists a constant $C>0$ depending on $k$ such that $[n,m]~\rightarrow~ \left[k, 2^{k-1}\right]$ for $m\in   \left[C\cdot2^{n\left(1-\frac{1}{2^{k-2}}\right)}, 2^n-C\cdot 2^{n\left(1-\frac{1}{2^{k-2}}\right)}\right]$.}
	
	\begin{proof}
		Note that we can assume $n$ to be sufficiently large when needed (by a suitable choice of $C$).
		
		Let $|S|=m$. The number of ordered pairs $(s,s^{*})$ such that $s\in S$ and $s^{*}\in \overline{S}$ is $m(2^n-m)$. For such pairs, the difference $s-s^{*}$ can take on at most $2^n-1$ possible values, and so there exists a difference $d$ that appears in at least $\frac{m(2^n-m)}{2^n-1}\ge \frac{m(2^n-m)}{2^n}$ ways as $s-s^{*}$. Let $d$ be such a difference, and consider the set $S_1\subseteq \ff_2^n/\langle d\rangle$ consisting of those $\langle d\rangle$-cosets that contain exactly one element of $S$. Then we have $|S_1|\ge \frac{m(2^n-m)}{2^n}$.
		
		If we can find a $(k-1)$-dimensional affine subspace of $\ff_2^n/\langle d\rangle$ that consists only of elements in $S_1$, then the lifting of this subspace to $\ff_2^n$ gives rise to a $[k,2^{k-1}]$-flat induced by $S$, which is what we need.
		
		By Theorem \ref{order2}, we have $[N, M]\to [K,2^K]$ for all integers $N\ge K\ge 1$ and $M\ge \frac52\cdot 2^{N\left(1-\frac{1}{2^{K-1}}\right)}$.
		
		Using this fact for $N=n-1$ and $K=k-1$, what remains to prove is that $|S_1|\ge \frac52\cdot 2^{(n-1)\left(1-\frac{1}{2^{k-2}}\right)}$.
		
		Letting $d=\left|m-\frac12\cdot 2^n\right|$, we have $m(2^n-m)=\left(\frac12\cdot 2^n-d\right)\left(\frac12\cdot 2^n+d\right)=\frac14\cdot 2^{2n}-d^2$. So it suffices to show that
		
		$$\frac{\frac14\cdot 2^{2n}-d^2}{2^n}\ge \frac52\cdot 2^{(n-1)\left(1-\frac{1}{2^{k-2}}\right)}$$
		
		$$\Leftrightarrow d^2\le \frac14\cdot 2^{2n}-c\cdot 2^{n\left(2-\frac{1}{2^{k-2}}\right)}$$
		
		for a positive constant $c\in \left(\frac54, \frac52\right]$ depending on $k$.
		
		As we have $d\le \frac12\cdot 2^n-C\cdot 2^{n\left(1-\frac{1}{2^{k-2}}\right)}$, this means that $$d^2\le \frac14\cdot 2^{2n}-C\cdot 2^{n\left(2-\frac{1}{2^{k-2}}\right)}+C^2\cdot 2^{n\left(2-\frac{1}{2^{k-3}}\right)}.$$ For an appropriate choice of $C$ and sufficiently large $n$, this expression is less than $$\frac14\cdot 2^{2n}-c\cdot 2^{n\left(2-\frac{1}{2^{k-2}}\right)}.$$
	\end{proof}
	
	The following theorem is a more specific form of Theorem \ref{2hatvany}:
	
	\begin{theorem}\label{induced_power2_flat}
		For every pair $(k,\ell)$ of integers such that $k>\ell\ge 1$, there exist constants $C>0$ and $0<D<1$ depending on $k,\ell$ such that if we are given a set $S\subseteq \ff_2^n$ of size between $C~\cdot~ 2^{n\left(1-\frac{1}{2^{k-\ell-1}}\right)}$ and $D\cdot 2^n$, then this set will contain an induced $\left[k, 2^{k-\ell}\right]$-flat. Moreover for each $\eps>0$, $D$ can take the value $\frac{1-\eps}{2^{\ell-1}}$ (with appropriate $C$).
	\end{theorem}
	\begin{proof}

		Let us choose directions $v_1, v_2, ..., v_{\ell}$ iteratively such that each direction $v_d$ forms a linearly independent set together with the previous ones, and the following condition holds for each $1\le~d\le~\ell$ (where $m=|S|$):
		\begin{equation}\label{eq:*}
			F_{d,1}(v_1,v_2,...,v_d)\ge \frac{m(2^n-m)(2^{n-1}-m)...(2^{n-d+1}-m)}{2^n\cdot 2^{n-1}\cdot ...\cdot 2^{n-d+1}}. 
		\end{equation}

		First we choose $v_1$. In the proof of Theorem \ref{halfflat}, we have already seen that it can be chosen such that $F_{1,1}(v_1)\ge \frac{m(2^n-m)}{2^n}$, satisfying (\ref{eq:*}) for $d=1$.
		
		Now suppose that $v_1, ..., v_d$ have already been chosen satisfying (\ref{eq:*}), where $1\le d\le \ell-1$. Then letting $L_d=\langle v_1, ..., v_d\rangle$, consider the $d$-flats parallel to $L_d$. There are $2^{n-d}$ such $d$-flats, and since $|S|=m$, at most $m$ of those contain some point of $S$. So the number of $d$-flats parallel to $L_d$ having empty intersection with $S$ is at least $2^{n-d}-m$.
		
		Any pair of flats parallel to $L_d$ forms a $(d+1)$-flat, and such a $(d+1)$-flat has exactly 1 point in $S$ if and only if one of the $d$-flats contained 1 point of $S$ and the other contained 0. So the number of such $[d+1,1]$-flats is $F_{d,1}(v_1,...,v_d)F_{d,0}(v_1,...,v_d)$. There are $2^{n-d}-1$ possible directions for such a flat (that is, the difference of the two flats in $\ff_2^n/L_d$ can have $2^{n-d}-1$ possible values), so there is a direction $v_{d+1}+L_d$ such that the number of $[d+1,1]$-flats parallel to $L_{d+1}=\langle v_1, ..., v_{d+1}\rangle$ is at least
		
		$$\frac{F_{d,1}(v_1,...,v_d)F_{d,0}(v_1,...,v_d)}{2^{n-d}-1}\ge \frac{F_{d,1}(v_1,...,v_d)\cdot (2^{n-d}-m)}{2^{n-d}}$$
		
		$$\ge \frac{m(2^n-m)(2^{n-1}-m)...(2^{n-d+1}-m)(2^{n-d}-m)}{2^n\cdot 2^{n-1}\cdot ...\cdot 2^{n-d+1}\cdot 2^{n-d}},$$
		
		where we used $(\ref{eq:*})$ for $d$ for the last inequality. This proved $(\ref{eq:*})$ for $d+1$.
		
		Now that $v_1, ..., v_{\ell}$ have been chosen, let $L_{\ell}=\langle v_1, ...., v_{\ell}\rangle$. Let $S'\subseteq \ff_2^n/L_{\ell}$ consist of those $L_{\ell}$-cosets which contain exactly 1 point of $S$.
		
		Again let us use Theorem \ref{order2}, which stated that $[N, M]\to [K,2^K]$ for all integers $N\ge K\ge 1$ and $M\ge \frac52\cdot 2^{N\left(1-\frac{1}{2^{K-1}}\right)}$. Using this for $N=n-\ell$ and $K=k-\ell$ in relation to $\ff_2^n/L_{\ell}$ and its subset $S'$, we would like to obtain a full $(k-\ell)$-flat consisting only of points in $S'$. By lifting such a flat to $\ff_2^n$, we get a $k$-flat that has $2^{k-\ell}$ points of $S$, which is just what we need.
		
		To do this, we need to ensure that $|S'|\ge \frac52\cdot 2^{(n-\ell)\left(1-\frac{1}{2^{k-\ell-1}}\right)}=A\cdot 2^{n\left(1-\frac{1}{2^{k-\ell-1}}\right)}$ for some constant $A$ depending on $k$ and $\ell$.
		
		If we have $m\le (1-\eps)2^{n-\ell+1}$ (corresponding to $D=\frac{1-\eps}{2^{\ell-1}}$), then by $(\ref{eq:*})$,
		
		$$|S'|=F_{\ell, 1}(v_1, ..., v_{\ell})\ge \frac{m(2^n-m)(2^{n-1}-m)...(2^{n-\ell+1}-m)}{2^n\cdot 2^{n-1}\cdot ...\cdot 2^{n-\ell+1}}=m\eps \prod_{k=1}^{\ell-1} \left(1-\frac{1-\eps}{2^k}\right)\ge m\eps \phi\left(\frac12\right)$$
		
		where $\phi\left(\frac12\right)$ is the  value of the Euler function at $\frac12$, which is approximately  $0.2888$.
		
		This can be ensured to be at least $A\cdot 2^{n\left(1-\frac{1}{2^{k-\ell-1}}\right)}$ by taking $m\ge \frac{A}{\eps\cdot \phi\left(\frac12\right)}\cdot 2^{n\left(1-\frac{1}{2^{k-\ell-1}}\right)}$.
		
		Therefore, if we have $\frac{A}{\eps\cdot \phi\left(\frac12\right)}\cdot 2^{n\left(1-\frac{1}{2^{k-\ell-1}}\right)}\le m\le \frac{1-\eps}{2^{\ell-1}}\cdot 2^n$, then $[n,m]\to [k,2^{k-\ell}]$.
	\end{proof}
	
	\begin{remark}Note that Theorem \ref{induced_power2_flat} does not hold for $D\ge \frac{1}{2^{\ell-1}}$, as taking $S$ to be a full $(n-\ell+1)$-flat in $\ff_2^n$, it will not induce any $[k, 2^{k-\ell}]$-flats, as for any $k$-flat $\mathcal{F}_k$ with $\mathcal{F}_k\cap S\ne\emptyset$, we must have $\dim \mathcal{F}_k-\dim (\mathcal{F}_k\cap S)\le \ell-1$, meaning that $|\mathcal{F}_k\cap S|\ge 2^{k-\ell+1}$.
	\end{remark}
	
	\section{Finding $[k,3\cdot 2^{k-\ell}]$-flats for almost all values $|S|$ when $S$ has a low positive density}\label{3power2_flats_section}
	
	We continue our investigation for the case when the binary form of $t$ contains exactly two  $1$s which are consecutive. The main focus will be the proof of Theorem \ref{3times2power}, which we will recall:
	
	\textbf{Theorem \ref{3times2power}.} \textit{
		For every pair $(k,\ell)$ of integers with $2\le \ell\le k-1$, the density of integer values $m$ within the interval $ \left[0, \frac{1}{2^{\ell-1}}\cdot 2^n\right]$ for which  $[n,m]\to [k,3\cdot 2^{k-\ell}]$ holds, tends to $1$ as $n\to \infty$. Hence $\rho(n; k, 3\cdot 2^{k-\ell})\geq \frac{1}{2^{\ell-1}}.$
	}
	
	The proof builds on ideas similar to those used for Theorem \ref{induced_power2_flat}, this time bounding the number of $[2,3]$-flats in $S$ based on its {additive energy}, which is an additive combinatorial notion that we will now define.
	
	For a set $S\subseteq \ff_2^n$, let  the \textit{additive energy} of $S$ (a notion introduced by Tao and Vu \cite[Section 2.3]{TaoVu}), denoted by $E(S)$ be as follows:
	
	$$E(S)=\{(u_1,u_2,u_3,u_4)\in S^4: u_1+u_2=u_3+u_4\}.$$
	
	If $|S|=m$ then $E(S)\le m^3$ since to each triple $(u_1,u_2,u_3)$, at most one value $u_4$ can belong, and there is equality if and only if $S$ is a flat of $\ff_2^n$ (which is easy to verify using characteristic $2$). However, this can only happen if $m$ is a power of $2$.
	
	In fact, the bound $E(S)\le m^3$ can be improved in general using the concept below.
	
	\begin{definition}
		For $d\in \zz^{+}$ and $a\in \zz$, define the \textit{least absolute residue} of $a$ modulo $d$, denoted $r_a(d)$, by
		
		$$r_a(d)=\min_{k\in \zz} |a-kd|.$$
	\end{definition}
	
	\begin{proposition}\label{energybound}
		Fix real numbers $\frac12\le\alpha<1$ and $\eps>0$ with $\alpha+\eps<1$. Then there exists $m_0\in \zz^{+}$ such that for any pair of integers $n>k\in \zz^{+}$, if $S\subseteq \ff_2^n$ is any set with $$m_0\le |S|=m\in [2^k, 2^{k+1}] \mbox{\  and \ } r_{2^{\lceil {(\alpha+\eps)k \rceil}}}(m)\ge 2^{\alpha k},$$ then $E(S)\le m^3-m^{2+\alpha-\eps}$.
	\end{proposition}
	
	Proving this bound requires estimates on sizes of cuts in hypercube graphs (and disjoint unions thereof), which we will defer to Section \ref{hypercube_cuts_section}, and the bound itself will be proven in Section \ref{energy_bound_section}. For now, let us apply this to prove our theorem.
	
	\subsection{Proof of Theorem \ref{3times2power}}
	
	The following lemma gives an alternative expression of the additive energy which will be helpful.
	
	\begin{lemma}
		For a nonzero vector $v\in \ff_2^n$, let us denote by $p_v$ the number of pairs in $S$ with difference $v$. Then we have 
		
		$$E(S)=|S|^2+4\sum_{v\in \ff_2^n\setminus \{0\}} p_{v}^2.$$
	\end{lemma}
	
	\begin{proof}
		The first term comes from quadruples of the form $(u_1,u_1,u_3,u_3)$ for arbitrary $u_1,u_3\in S$. For the second term we can choose any pairs $\{u_1,u_2\}$ and $\{u_3,u_4\}$ with sum $u_1+u_2=u_3+u_4=v$ and then order each pair in one of two ways.
	\end{proof}
	
	Now we can relate the additive energy to the number of induced $[2,3]$-flats:
	
	\begin{lemma}\label{23flats_energy}
		Let $S\subseteq \ff_2^n$ be a set with $|S|=m$. Then the number of $2$-flats of $\ff_2^n$ containing exactly 3 points of $S$ (denoted $F_{2,3}$) satisfies $F_{2,3}=\frac16(m^3-E(S))$.
	\end{lemma}
	\begin{proof}
		Taking a nonzero vector $v\in \ff_2^n$, we clearly have $F_{1,2}(v)=p_{v}$ and $F_{1,1}(v)=m-2p_{v}$. A $[2,3]$-flat is composed of a $[1,2]$-flat and a $[1,1]$-flat which are parallel to each other, however each $[2,3]$-flat can be decomposed in 3 such ways, so $$3F_{2,3}=\sum_{v\ne 0} F_{1,2}({v})F_{1,1}({v})=\sum_{v\ne 0} p_{v}(m-2p_{v})=m\binom{m}{2}-2\sum_{v\ne 0} p_{v}^2,$$ where $\sum_{v\ne 0} p_{v}^2=\frac{E(S)-m^2}{4}$. This means that $3F_{2,3}=\frac{m^2(m-1)}{2}-2\cdot\frac{E(S)-m^2}{4}=\frac12m^3-\frac12E(S)$, giving the statement.
	\end{proof}
	
	\begin{proof}[Proof of Theorem \ref{3times2power}]
		We follow a similar proof method as for Theorem \ref{induced_power2_flat}.
		
		Fix an $\eps_0>0$ and take a set $S\subseteq \ff_2^n$ with $|S|=m\le (1-\eps_0)2^{n-\ell+1}$.
		
		Let us choose directions $v_1, v_2, ..., v_{\ell}$ so that each direction $v_d$ forms a linearly independent set with the previous ones, and the following condition holds for each $d\in \{2, 3, \ldots, \ell\}$:
		
		$$F_{d,3}(v_1,v_2,...,v_d)\ge \frac{F_{2,3}}{\twobinom{n}{2}}\cdot \frac{(2^{n-2}-m)(2^{n-3}-m)...(2^{n-d+1}-m)}{2^{n-2}\cdot 2^{n-3}\cdot ...\cdot 2^{n-d+1}}.$$
		
		This can be done in the same way as in Theorem \ref{induced_power2_flat}, except for the following first step: for $d=2$, choose $v_1$ and $v_2$ together such that for $L_2=\langle v_1,v_2\rangle$, the number of $[2,3]$-flats parallel to $L_2$ is at least ${F_{2,3}}{\Large/}{\footnotesize\twobinom{n}{2}}$. (This can be done, as there are $\footnotesize\twobinom{n}{2}$ possible directions of $2$-flats.)
		
		After such directions have been chosen, in order to use Theorem \ref{order2} to get a $[k, 3\cdot 2^{k-\ell}]$-flat, we need the following condition to hold:
		
		$$\frac{F_{2,3}}{\twobinom{n}{2}}\cdot \frac{(2^{n-2}-m)(2^{n-3}-m)...(2^{n-\ell+1}-m)}{2^{n-2}\cdot 2^{n-3}\cdot ...\cdot 2^{n-\ell+1}}\ge \frac52\cdot 2^{(n-\ell)\left(1-\frac{1}{2^{k-\ell-1}}\right)}.$$
		
		Since $m\le (1-\eps_0)2^{n-\ell+1}$, the left hand side is at least $\frac{6F_{2,3}}{2^{2n}}\eps_0\phi\left(\frac12\right)$. Here we used the fact that $\twobinom{n}{2}=\frac{(2^n-1)(2^n-2)}{6}$. Therefore, it suffices to require
		
		$$F_{2,3}\ge \lambda_1\cdot 2^{n\left(3-\frac{1}{2^{k-\ell-1}}\right)}$$
		
		for a fixed positive constant $\lambda_1$ depending on $k$ and $\ell$ only.
		
		Now if $\delta=\frac{1}{2^{k-\ell-1}}$, let $\alpha=1-\frac{\delta}{4}$ and $\eps=\frac{\delta}{8}$. Then apply Proposition \ref{energybound} for these values $\alpha$ and $\eps$, gaining that for sufficiently large $m$ satisfying $r_{2^{\lceil (\alpha+\eps)t\rceil}}(m)\ge 2^{\alpha t}$ where $m\in [2^t, 2^{t+1}]$, we have $E(S)\le m^3-m^{3-\frac38\delta}$.
		
		By Lemma \ref{23flats_energy}, we then have $F_{2,3}=\frac16(m^3-E(S))\ge \frac16m^{3-\frac38\delta}$. For the guaranteed existence of a $[k, 3\cdot 2^{k-\ell}]$-flat by our method, it suffices to have $\frac16m^{3-\frac38\delta}\ge \lambda_1\cdot 2^{(3-\delta)n}$, equivalent to $m\ge \lambda_2\cdot \left(2^n\right)^{\frac{3-\delta}{3-\frac38\delta}}$ for a positive constant $\lambda_2$. Letting $\mu\in \left(\frac{3-\delta}{3-\frac38\delta},1\right)$ be fixed, we can freely assume that $m\ge 2^{\mu n}$ since the density of values $m$ below this value in the range $[0, \frac{1-\eps_0}{2^{\ell-1}}\cdot 2^n]$ tends to 0. Hence for $n$ sufficiently large, we will then have $m\ge \lambda_2\cdot \left(2^n\right)^{\frac{3-\delta}{3-\frac38\delta}}$. 
		
		Then taking an interval $[2^t, 2^{t+1})$, if $m$ lies uniformly in this interval then the probability that $r_{2^{\lceil (\alpha+\eps)t\rceil}}(m)\ge 2^{\alpha t}$ is $\ge 1-\frac{2\cdot \left\lfloor 2^{\alpha t}\right\rfloor +1}{2^{\lceil (\alpha+\eps)t\rceil}}\ge 1-\frac{2\cdot 2^{\alpha t}+1}{2^{(\alpha+\eps)t}}=1-\frac{2}{2^{\eps t}}-\frac{1}{2^{(\alpha+\eps)t}}$. With the assumption that $m\ge 2^{\mu n}$, we have $t\ge \mu n$, hence this lower bound for the probability of correctness of $m$ tends uniformly to $1$ as $n\to \infty$.
		
		So for each $\eps_0>0$, the density of integers $m\in \left[0,\frac{1-\eps_0}{2^{\ell-1}}\cdot 2^n\right]$ such that $[n,m]\to [k,3\cdot 2^{k-\ell}]$ tends to $1$ as $n\to \infty$. Therefore the same holds for $m\in \left[0,\frac{1}{2^{\ell-1}}\cdot 2^n\right]$ as well.
	\end{proof}
	
	\subsection{Bounding sizes of hypercube cuts}\label{hypercube_cuts_section}
	
	\begin{definition}
		The $n$-dimensional hypercube graph $Q_n$ is a graph with vertex set $\ff_2^n$ such that two vertices are connected by an edge if and only if their Hamming distance is 1.
	\end{definition}
	
	\begin{definition}
		Define the function $\Psi: \mathbb{N}^{+}\to \mathbb{N}$ by $\Psi(t)=\sum\limits_{i=0}^{t-1} s_2(i)$. Recall that $s_2$ denotes the binary digit sum function.
	\end{definition}
	
	Observe that for  any positive integer $k$ we have $\Psi(2^k)=k\cdot 2^{k-1}$, as for each $1\le j\le k$, precisely half of the integers between $0$ and $2^k-1$ have a digit $1$ at the $j^{\mathrm{th}}$ position in the binary representation.
	
	\begin{lemma}[Hart, \cite{Hart}]\label{Hart1}
		For a positive integer $t$, let $c(t)$ be the maximum number of edges in a $t$-vertex subgraph of $Q_n$ ranging over all possible values $n\in \mathbb{N}^{+}$. Then for each $t$ we have $c(t)\le \Psi(t)$.
	\end{lemma}
	See \cite{Hart} for a proof and a description of the subgraphs attaining the upper bound for $c(t)$. Notably, one such optimal subgraph is induced by the lexicographic construction $\mathcal{L}_t$, introduced in Definition \ref{lexico}.  For any $T\subseteq V(Q_n)$  we have the simple relation $$2e(T)+e(T,\overline{T})=n|T|$$ by double counting all incident edges to each point of $T$ \cite{Hart}. This yields the corollary below.
	
	\begin{corollary}[Hart, Theorem 1.5, \cite{Hart}]\label{Hart2}
		Let $n\ge 1$ and $1\le t\le 2^n-1$ be fixed integers. Then the minimum possible value of $e(T, \overline{T})$ for a subset $T\subseteq V(Q_n)$ of size $t$ is $nt-2\Psi(t)$.
	\end{corollary}
	
	From this, we derive the following consequence:
	
	\begin{proposition}\label{cutcrossing_edges}
		Let $n\ge 1$ and $0\le d\le n-1$ be integers. Take any cut  $T\cup \overline{T}=\ff_2^n$ of the $n$-dimensional hypercube graph $Q_n$ that has at least $2^d$ points in both parts. Then at least $2^d(n-d)$ edges of $Q_n$ cross the cut, i.e., $e(T, \overline{T})\ge 2^d(n-d)$. 
	\end{proposition}
	\begin{proof}
		Suppose that $t=|T|$ and $2^d\le t\le 2^{n-1}$. Using Corollary \ref{Hart2}, we only need to show that $nt-2\Psi(t)\ge 2^d(n-d)$ holds.
		
		It is elementary to see that for integers $0\le k\le n-1$, the function $2^k(n-k)$ is increasing. This implies that in the case $t=2^{n-1}$, the statement holds as $nt-2\sum_{i=0}^{t-1} s_2(i)=n\cdot 2^{n-1}-2(n-1)2^{n-2}=2^{n-1}=2^{n-1}(n-(n-1))\ge 2^d(n-d)$.\\ So from now on, assume that $2^d\le t<2^{n-1}$. The monotonicity implies that
		we can assume  $2^d\le t<2^{d+1}$, where $d\le n-2$.
		
		The function $\Psi$ can be related to the well-known \textit{Takagi function} $\tau$, which is a continuous but nowhere differentiable non-negative function on $[0,1]$.   Let $t\in \zz^{+}$  be expressed as $t=2^d(1+x)$ for $d\in \mathbb{N}$ and $x\in [0,1)$.  Then  \begin{equation}\label{eq:Taka}
			2\Psi(t)=td+2^d(2x-\tau(x)),		
		\end{equation}
		see Monroe \cite[Corollary 5.0.8 (a)]{Monroe}. For a survey on the Takagi function, see \cite{Lagarias}.
		
		Now $nt-2\Psi(t)\ge 2^d(n-d) \Leftrightarrow 2\Psi(t)\le n(t-2^d)+d\cdot 2^d$. Since $n\ge d+2$, it suffices to prove that $2\Psi(t)\le (d+2)(t-2^d)+d\cdot 2^d$. By (\ref{eq:Taka}), this is equivalent to $$td+2^d(2x-\tau(x))\le (d+2)(t-2^d)+d\cdot 2^d \Leftrightarrow 2^{d+1} x-2^d\cdot \tau(x)\le 2\cdot 2^d x,$$ which holds since $\tau(x)\ge 0$ for all $x$.
	\end{proof}
	
	Let us prove a slightly different version of this proposition, for the case of multiple copies of the hypercube graph.
	
	\begin{proposition}\label{cutcrossing_edges_2}
		Let $n\ge 1$, $k\ge 1$ and $0\le d\le n-2$ be integers. Let $kQ_n$ denote the disjoint union of $k$ copies $Q_n^{(1)}$, $Q_n^{(2)}$, ..., $Q_n^{(k)}$ of the hypercube graph $Q_n$. Take any cut $kQ_n=A\sqcup B$ such that $|A|\ge 2^d$, and for each $1\le i\le k$, $1\le |A\cap Q_n^{(i)}|\le 2^{n-1}$. Then at least $2^d(n-d)$ edges of $kQ_n$ cross the cut.
	\end{proposition}
	
	\begin{proof}
		By \cite[Proposition 2.1]{Hart}, the function $\Psi$ satisfies the recursion $$\Psi(t)=\max\limits_{0<m\le \frac{t}{2}} (\Psi(m)+\Psi(t-m)+m),$$ hence for any partition $t=t_1+t_2$ with $t_1,t_2\in \zz^{+}$, we have $\Psi(t)\ge \Psi(t_1)+\Psi(t_2)$. By induction, this implies that for any partition $t=t_1+t_2+...+t_k$ (for any $k\ge 1$) with $t_1,t_2,...,t_k\in \zz^{+}$, we have $\Psi(t)\ge \Psi(t_1)+\Psi(t_2)+...+\Psi(t_k)$.
		
		Let $A=\bigsqcup\limits_{i=1}^k A_i$ where $A_i=A\cap Q_n^{(i)}$, and similarly let $B=\bigsqcup \limits_{i=1}^k B_i$.
		
		We consider two cases depending on $\max\limits_{1\le i\le k} |A_i|$:
		
		\textbf{Case 1.} There exists an $i$ such that $|A_i|\ge 2^d$.
		
		In this case, in $Q_n^{(i)}$ we have a cut into $A_i$ and $B_i$ such that $2^d\le |A_i|\le 2^{n-1}$, which, by Proposition \ref{cutcrossing_edges}, intersects at least $2^d(n-d)$ edges on its own.
		
		\textbf{Case 2.} For all values $i$, we have $|A_i|<2^d$.
		
		In this case, each $i$ satisfies $|A_i|<2^d\le 2^{n-2}$. Letting $\sigma_j=\sum_{i=1}^j |A_i|$ for each $1\le j\le k$, we have $\sigma_k=|A|\ge 2^d$. Let $j\in \{1, 2, \ldots, k\}$ be the smallest value such that $\sigma_j\ge 2^d$. Then $\sigma_{j-1}<2^d$, and $\sigma_j=\sigma_{j-1}+|A_j|<2^d+2^d\le 2^{n-1}$. Now adding up the number of edges crossing the cut in the first $j$ hypercubes, we get at least \begin{equation}
			\begin{split}
				n|A_1|-2\Psi(|A_1|)+n|A_2|-2\Psi(|A_2|)+...+n|A_j|-2\Psi(|A_j|)\ge \\ n(|A_1|+|A_2|+...+|A_j|)-2\Psi(|A_1|+|A_2|+...+|A_j|)=n\sigma_j-2\Psi(\sigma_j)
			\end{split}
		\end{equation} by Corollary \ref{Hart2}.
		
		This in turn completes the proof, as $n\sigma_j-2\Psi(\sigma_j)\ge 2^d(n-d)$ by the proof of Proposition \ref{cutcrossing_edges}, since $j$ was chosen so that 
		$2^d\le \sigma_j\le 2^{n-1}$ holds.
	\end{proof}
	
	\subsection{Bounding the additive energy}\label{energy_bound_section}
	
	\begin{proposition}\label{manyrichdiff}
		Given integers $m,b,c,d,n\ge 1$ with $d<b$ and $r_{2^{b+1}}(m)\ge 2^d$, the following holds. For any $S\subseteq \ff_2^n$ with $|S|=m$, if there are at least $2^b$ distinct nonzero differences $v$ for $S$ with $p_{v}\ge c$, then $(b+1)(m-2c)\ge 2^d(b+1-d)$.
	\end{proposition}
	\begin{proof}
		Take the linear span of the $2^b$ distinct differences. This span must have dimension at least $b+1$ (otherwise the number of vectors would be at most $|\ff_2^{b}\setminus \{0\}|=2^b-1$). Pick $b+1$ linearly independent differences $(v_1, v_2, ..., v_{b+1})$ with $p_{v_i}\ge c$ for all $1\le i\le b+1$, and let $L=\left\langle v_1,v_2, ..., v_{b+1}\right\rangle$. Consider the $L$-cosets of $\ff_2^n$ based on how many elements of $S$ they contain. A coset containing $2^{b+1}$ or $0$ elements of $S$ will be called \textit{full} or \textit{empty} respectively. The remaining elements of $S$ (not in full cosets) are spread across some cosets (say, $\ell$ of them): let $U$ be the union of these cosets. Since $r_{2^{b+1}}(m)\ne 0$, we must have $\ell\ge 1$. On the vertex set $U$, define the graph with edges joining all pairs of vertices with a difference of $v_i$ for any $1\le i\le b+1$. This graph is isomorphic to $\ell Q_{b+1}$; let the $i$-th copy of $Q_{b+1}$ be $Q_{b+1}^{(i)}$. For each $i$, define a partition of $Q_{b+1}^{(i)}$ into sets $A_i$ and $B_i$ as follows: if $|Q_{b+1}^{(i)}\cap S|\le \frac12\cdot 2^{b+1}$ then let $A_i=Q_{b+1}^{(i)}\cap S$ and $B_i=Q_{b+1}^{(i)}\setminus S$, and otherwise let $B_i=Q_{b+1}^{(i)}\cap S$ and $A_i=Q_{b+1}^{(i)}\setminus S$. (This way, for each $1\le i\le \ell$ we will have $|A_i|\le |B_i|$.) Let $A=\cup_{i=1}^{\ell} A_i$ and $B=\cup_{i=1}^{\ell} B_i$. Thus we obtained a partition $U=A\sqcup B$.
		
		Since for each $i$, we have $p_{v_i}\ge c$, this means that at most $m-2c$ elements $w\in S$ satisfy the property that $w+v_i\not\in S$. This means that altogether (summing over all $1\le i\le b+1$), the number of pairs $(w,w')$ with $w\in S$ and $w'\not\in S$ such that $w'-w=v_i$ for some $1\le i\le b+1$, which is the same as $e(U\cap S, U\setminus S)=e(A,B)$, is at most $(b+1)(m-2c)$. 
		
		Without loss of generality, let $S\cap U=A_1\sqcup A_2\sqcup ...\sqcup A_{\ell'}\sqcup B_{\ell'+1}\sqcup ... \sqcup B_{\ell}$. Now if the number of full cosets is $f$ then 
		\begin{equation}
			\begin{split}
				m&=f\cdot 2^{b+1}+|A|+\sum_{i=\ell'+1}^{\ell} (|B_i|-|A_i|)=f\cdot 2^{b+1}+|A|+\sum_{i=\ell'+1}^{\ell} (2^{b+1}-2|A_i|)\\ &=z\cdot 2^{b+1}+\sum_{i=1}^{\ell'} |A_i|-\sum_{i=\ell'+1}^{\ell} |A_i|
			\end{split}
		\end{equation}
		
		for some integer $z$. Hence $2^d\le r_{2^{b+1}}(m)\le |m-z\cdot 2^{b+1}|\le \sum_{i=1}^{\ell} |A_i|=|A|$.
		
		Let us now apply Proposition \ref{cutcrossing_edges_2} with $n=b+1$ and $k=\ell$, using the conditions $|A|\ge 2^d$ and $d\le b-1=n-2$. The number of crossing edges is at least $2^d(n-d)=2^d(b+1-d)$ and, as seen previously, at most $(b+1)(m-2c)$. This completes the proof.
	\end{proof}
	
	\begin{lemma}\label{sumofsquaresbound}
		Let $(h_1,h_2,...,h_{\alpha})$ be a sequence of non-negative real numbers with $\sum_{i=1}^{\alpha} h_i=N$. Suppose that $0\le \alpha'\le \alpha$ for $\alpha' \in \mathbb{Z}$ and $0<M_1<M$ for $M_1,M\in \mathbb{R}$.
		
		If $h_i\le M$ holds for all $i$, and $h_i>M_1$ for at most $\alpha'$ values of $i$, then $$\sum_{i=1}^{\alpha} h_i^2\le \alpha'M^2+(N-\alpha'M)M_1.$$
	\end{lemma}
	\begin{proof}
		Since $\alpha-\alpha'$ of the values lie in $[0,M_1]$ and the others in $[0,M]$, we have $$\sum_{i=1}^{\alpha} \left(h_i-\frac{M_1}{2}\right)^2\le \alpha'\left(M-\frac{M_1}{2}\right)^2+(\alpha-\alpha')\left(\frac{M_1}{2}\right)^2=\alpha'M^2-\alpha'MM_1+\alpha\frac{M_1^2}{4}.$$
		
		This means that $\sum_{i=1}^{\alpha} h_i^2=\sum_{i=1}^{\alpha} \left(h_i-\frac{M_1}{2}\right)^2+NM_1-\alpha\frac{M_1^2}{4} \le \alpha'M^2-\alpha'MM_1+NM_1$ holds, as required.
	\end{proof}
	
	Now we are ready to prove Proposition \ref{energybound}.
	
	\begin{proof}[Proof of Proposition \ref{energybound}]
		Throughout the proof, we will use the fact that $m$, and hence $k$, are taken to be sufficiently large.
		
		In order to bound the additive energy, we will use the expression $$E(S)=|S|^2+4\sum\limits_{{v}\in \ff_2^n\setminus \{0\}} p_{v}^2.$$ Firstly, it is clear that for any ${v}\ne 0$, we have $p_{v}\le \frac12m$. Set $c=\frac12m-\frac{\eps}{4(\alpha+\eps)}2^{\alpha k}$. Let us put an upper bound on the number of differences ${v}\in \ff_2^n\setminus \{0\}$ such that $p_{v}\ge c$. Let $\Delta$ be the number of such differences.
		
		In order to do this, we will apply Proposition \ref{manyrichdiff}. Via replacing $c$ by $\lceil c\rceil$, it is clear that the statement of this proposition holds for real numbers $c$ as well. 
		
		Fix $b=\lceil (\alpha+\eps)k\rceil$. Then we have $r_{2^b}(m)\ge 2^{\alpha k}$ and hence $r_{2^{b+1}}(m)\ge 2^{\alpha k}$ too. Set $d=\lfloor \alpha k\rfloor$; then if $\Delta\ge 2^b$ then the conditions of the Proposition hold, so we have $(b+1)(m-2c)\ge 2^d(b+1-d)$.
		
		Via differentiation, it can be seen that the function $d\mapsto 2^d(b+1-d)$ is increasing on the real interval $[0,b-1]$, therefore we have $2^d(b+1-d)\ge 2^{\alpha k-1}(b+1-(\alpha k-1))$ if $0\le \alpha k-1$ and $\lfloor \alpha k\rfloor\le b-1$, which is equivalent to $1\le \alpha k<b$, which is true for $k\ge 2$.
		
		So we have

		$$(b+1)(m-2c)\ge 2^{\alpha k-1}(b+1-(\alpha k-1)).$$
		
		Using $(\alpha+\eps)k\le b\le (\alpha+\eps)k+1$:
		
		$$((\alpha+\eps)k+2)\frac{\eps}{2(\alpha+\eps)} 2^{\alpha k}\ge 2^{\alpha k-1}((\alpha+\eps)k-\alpha k+2)$$
		
		$$\Leftrightarrow\frac{\eps}{\alpha+\eps}\ge \frac{\eps k+2}{(\alpha+\eps)k+2},$$

		would be a contradiction. Therefore we must have $\Delta<2^b$, so the number of differences such that $p_v\ge c$ is less than $2^b$.
		
		Therefore, by using Lemma \ref{sumofsquaresbound} for the $\alpha=2^n-1$ values $p_{v}$ with $N=\sum p_{v}=\binom{|S|}{2}=\frac{m(m-1)}{2}$, $~M=\frac12m$, $~M_1=c$ and $\alpha'=2^b$ (for $k>b$, we have $2^b<2^k\le m\le 2^n$, so the condition $\alpha'\le \alpha$ holds), we have 
		
		\begin{equation*} \label{eq1}
			\begin{split}
				\sum_{{v}} p_{v}^2 
				& \le 2^b\left(\frac12m\right)^2+\left(\frac{m(m-1)}{2}-2^b\cdot \frac12m\right)c
				\le 2^b\cdot \frac14m^2 +\left(\frac12m^2-2^b\cdot \frac12m\right)\left(\frac12m-\frac{\eps}{4(\alpha+\eps)}2^{\alpha k}\right)\\
				&=\frac14m^3-\frac{\eps}{8(\alpha+\eps)}m^2\cdot 2^{\alpha k}+\frac{\eps}{8(\alpha+\eps)}m\cdot 2^{\alpha k+b}
				\le \frac14m^3-\frac{\eps}{8(\alpha+\eps)} m^2\cdot\frac{m^{\alpha}}{2^{\alpha}}+\frac{\eps}{8(\alpha+\eps)} m\cdot 2m^{2\alpha+\eps}\\  
				&=\frac14m^3-\frac{\eps}{8\cdot 2^{\alpha}\cdot (\alpha+\eps)} m^{2+\alpha}+\frac{\eps}{4(\alpha+\eps)} m^{1+2\alpha+\eps},
			\end{split}
		\end{equation*}

		where $1+2\alpha+\eps<2+\alpha$ holds, since $\alpha+\eps<1$. Thus we obtain $m^{1+2\alpha+\eps}=o(m^{2+\alpha})$, and
		
		$$E(S)=m^2+4\sum\limits_{{v}\in \ff_2^n\setminus \{0\}} p_{v}^2\le m^3-\frac{\eps}{2^{\alpha+1}(\alpha+\eps)}m^{2+\alpha}+o(m^{2+\alpha}),$$
		
		showing the statement.
	\end{proof}
	
	\section{Concluding remarks}
	
	For general pairs $(k,t)$, it remained  open to decide whether the density of the spectrum $\rho(n;k,t)= \frac{Sp(n;k,t)}{2^n}$ tends to $1$ as $n$ tends to infinity, or even to obtain a general constant lower bound on this density for a given pair $(k,t)$ with $0\leq t \leq 2^k$. The proof of Theorem \ref{3times2power} suggests that such a result might be deduced from a supersaturation type result on $[k,t]$-flats, similar to the lower bound on $F_{2,3}$ derived from Proposition \ref{energybound}. 
	
	We did not elaborate on the case $q>2$. However, note that the approach we used to exclude several values from the spectra $Sp(n; k,t)$ applies for $q>2$ as well. Indeed, we can construct point sets as a 
	combination of disjoint full flats and point sets which hit every $k$-flat in small number of points, i.e., evasive sets. However, proving density results on the spectrum appear to be much harder, as it was the case even for $q=3$ with lines (or full flats), as the general problem contains the determination of the order of magnitude of the maximum size of $q$-AP-free sets of $\ff_q^n$.
	
	While we focused on the sets in $\AG(n,q)$, $q=2$, we mention that the analogous question was investigated by  Ihringer and Verstra\"ete \cite{Ferdinand} for point sets $X$  in a projective space  $\PG(n,q)$ where all $k$-dimensional subspaces must have a bounded number of points. However their results  concern  the case when $q$ is large.\\
	
	\noindent {\bf Acknowledgement}\\
	We would like to thank the anonymous referees for their valuable suggestions  which helped us in improving the paper.

\end{document}